\documentclass[a4paper,12pt,reqno]{amsart}
\usepackage{amsmath,amsfonts,amssymb,amsthm,enumerate,multicol}%hyperref,
\usepackage{tikz,graphicx}
\usepackage{hyperref}
\usepackage{caption,enumitem,wasysym}
\usepackage{cleveref}
\usepackage{epstopdf}
\usepackage{float,wrapfig}
\usepackage[labelsep=space]{caption}
\usepackage[font=footnotesize]{caption}
\usepackage{xcolor}

\newtheorem{theorem}{Theorem}[section]

\newtheorem{remark}{Remark}[section]

\newtheorem{proposition}{Proposition}[section]

\numberwithin{equation}{section}

\title{Spectral properties related to generalized complementary Romanovski-Routh polynomials}
\author{Vinay Shukla$^\dagger$}
\address{$^\dagger$Department of Mathematics\\ Indian Institute of Technology, Roorkee-247667, Uttarakhand, India}
\email{vshukla@ma.iitr.ac.in}
\author{A. Swaminathan\, $^{\#\ddagger}$}{\thanks{$^{\#}$Corresponding author}}
\address{$^\ddagger$Department of Mathematics\\ Indian Institute of Technology, Roorkee-247667, Uttarakhand, India}
\email{mathswami@gmail.com, a.swaminathan@ma.iitr.ac.in}
\allowdisplaybreaks
\bigskip

\begin{document}



\begin{abstract}
Complementary Romanovski-Routh polynomials play an important role in extracting specific properties of orthogonal polynomials. In this work, a generalized form of the Complementary Romanovski-Routh polynomials (GCRR) that has the Gaussian hypergeometric representation and satisfies a particular type of recurrence called $R_{II}$ type three term recurrence relation involving two arbitrary parameters is considered. Self perturbation of GCRR polynomials leading to extracting two different types of $R_{II}$ type orthogonal polynomials are identified. Spectral properties of these resultant polynomials in terms of tri-diagonal linear pencil were analyzed. The LU decomposition of these pencil matrices provided interesting properties involving biorthogonality. Interlacing properties between the zeros of the polynomials in the discussion are established.
\end{abstract}
\maketitle
\markboth{Vinay Shukla and A. Swaminathan}{Linear combination of $R_{II}$ polynomials}	

\section{Introduction}\label{Introduction}

Self perturbation of a sequence of orthogonal polynomials $\{\mathcal{P}_n(x)\}_{n=0}^\infty$ is defined as \cite{Kwon Lee_Paco Park 2001}
\begin{align}\label{Linear combination Q_n}
\mathcal{L}_n(x)=\mathcal{P}_n(x)-\alpha_n\mathcal{P}_{n-1}(x), \quad \alpha_n \in (\mathbb{R} ~ \rm{or} ~ \mathbb{C})\backslash \{ 0 \}, \quad n\geq 0,
\end{align}
where $\{\mathcal{L}_n(x)\}_{n=0}^\infty$ denotes the perturbed sequence. The study of such self perturbation of polynomials is useful because the linear span of polynomials so formed is rich in the sense of cardinality. \eqref{Linear combination Q_n} can also be viewed as a linear combination of two continuous terms of the sequence $\{\mathcal{P}_n(x)\}_{n=0}^\infty$. Linear combinations of orthogonal polynomials have been extensively explored in the literature, both from a theoretical and practical application point of view. In the fundamental paper \cite{Shohat TAMS 1937}, a linear combination of orthogonal polynomials arises naturally in the study of mechanical quadrature rules. In the context of quasi-orthogonality, a sequence of orthogonal polynomials written as a linear combination of a fixed number of elements from another sequence of orthogonal polynomials is investigated \cite{Chihara quasi PAMS 1957}. If a polynomial sequence $\{q_n(x)\}_{n=0}^\infty$ satisfies the condition \cite{Brezinski Driver Redivo-Zaglia ANM 2004}
\begin{align*}
\int_{a}^{b} x^r q_n(x)w(x)dx \begin{cases}
	= 0, \quad  r=0,1,\ldots, n-k-1, \\
	\neq 0 \quad r =n-k,
\end{cases}
\end{align*}
where $w(x)$ is a positive weight function on $[a,b]$, then $q_n(x)$ is said to be quasi-orthogonal of order $k$ on $[a,b]$ with respect to $w(x)$. The necessary and sufficient conditions for orthogonality of general linear combinations $\{q_n(x)\}_{n=0}^\infty$, say
\begin{align*}
q_n(x)=p_n(x)+a_1p_{n-1}(x)+\ldots++a_kp_{n-k}(x), \quad a_k \neq 0, \quad n \geq k,
\end{align*}
where $\{p_n(x)\}_{n=0}^\infty$ is a sequence of orthogonal polynomials, are examined in \cite{Alfaro Marcellan Pena Rezola JCAM 2010}. The concept of quasi-orthogonality was first introduced in \cite{Riesz 1923}. In \cite{Fejer 1933}, self perturbation (the case $k=1$) was studied, and the algebraic properties of such polynomials were analysed by many authors \cite{Draux Integral transform 2016, Joulak ANM 2005, Marcellan Peherstorfer Steinbauer 1996}. \par

Orthogonal polynomials on the real line generated by the recurrence relation
\begin{align*}
&\mathcal{B}_{n+1}(x) = (x-c_n)\mathcal{B}_n(x)-\lambda_n x\mathcal{B}_{n-1}(x), \quad n\geq 0, \quad x \in \mathbb{R}, \\
& \nonumber	\mathcal{B}_{-1}(x) = 0, \qquad \mathcal{B}_0(x) = 1,
\end{align*}
where $\{c_n\}_{n \geq 0}$ and $\{\lambda_n\}_{n \geq 0}$ are positive, are known as Laurent orthogonal polynomials or L-orthogonal polynomials. In \cite{Batelo Bracciali Ranga JCAM 2005}, a linear combination of these polynomials are studied by relating them to the L-orthogonal polynomials associated with a class of strong distribution functions $S^3(1/2, \beta, b)$, $0 < \beta < b <\infty$. \par

A sequence of monic orthogonal polynomials $\{\mathcal{P}_n(x)\}_{n=0}^\infty$ is $d$-orthogonal with respect to a vector of linear forms $\mathcal{U}=(u_o,\ldots,u_{d-1})^T $, if
\begin{align*}
&\langle u_r,x^m\mathcal{P}_n(x) \rangle	=0, \quad n\geq md+r+1, \quad m \geq 0,\\
& \langle u_r,x^m\mathcal{P}_{md+r}(x) \rangle \neq 0, \quad m \geq 0, \quad 0 \leq r \leq d-1.
\end{align*}
Self perturbation of a monic $d$-orthogonal polynomial sequence
%defined as
%\begin{align*}
%\mathcal{S}_n(x)=\mathcal{P}_{n+1}(x)+\alpha_{n+1}\mathcal{P}_{n}(x), \quad \alpha_n \in \mathbb{R}\backslash \{ 0 \}, \quad n\geq 0,
%\end{align*}
has been independently studied in \cite{Marcellan Saib 2019}. The conditions under which such a perturbed sequence, say $\{\mathcal{L}_n(x)\}_{n=0}^\infty$, becomes a monic $d$-orthogonal polynomial sequence are also obtained \cite{Marcellan Saib 2019}. Note that, for $d=1$, we retrieve the standard orthogonality. A particular case of $2$-orthogonal polynomials and their Darboux transformation have been discussed recently in \cite{Barrios Ardila RACSAM 2020, Marcellan Chaggara Ayadi 2021}.
The orthogonality of Hahn-Appell polynomial is established in \cite{Varma lmaz zarslan RACSAM 2019}.

The recurrence relation, given as
\begin{align*}
&  \mathcal{F}_{n+1}(z) = k_n(z-c_n)\mathcal{F}_n(z)-\lambda_n (z-a_n)\mathcal{F}_{n-1}(z), \quad n\geq 0, \\
& \nonumber	\mathcal{F}_{-1}(z) = 0, \qquad \mathcal{F}_0(z) = 1,
\end{align*}
with the assumptions $k_n \neq 0$ and $\lambda_n \neq 0$, is said to be of $R_I$ type \cite{Esmail masson JAT 1995} and polynomials generated by this relation are called $R_I$ polynomials. Recently, such a perturbed sequence of $R_I$ polynomials $\mathcal{L}_n(x)$, as defined by \eqref{Linear combination Q_n}, was inspected by recursively constructing a unique sequence $\{\alpha_{n}\}_{n=0}^\infty$  such that $\mathcal{L}_n(x)$ is again a sequence of $R_I$ polynomials \cite{KKB Swami 2018}. Biorthogonality properties and para-orthogonal polynomials obtained from self perturbed sequences are also reviewed. \par

Recurrence relation of the form
\begin{align}\label{R2 Recurrence Relation}
	&  \mathcal{P}_{n+1}(x) =\rho_n (x-c_n)\mathcal{P}_n(x)-d_n (x-a_n)(x-b_n)\mathcal{P}_{n-1}(x), \quad n\geq 0, \\
	& \nonumber	\mathcal{P}_{-1}(x) = 0, \qquad \mathcal{P}_0(x) = 1,
\end{align}
are well explored in the literature \cite{Esmail masson JAT 1995}. It was established that if $d_n \neq 0$ and $\rho_n \neq 0$ with $\mathcal{P}_n(a_n)\mathcal{P}_n(b_n)\neq 0$, for $n\geq 0$, then there exists a linear functional $\mathfrak{N}$ such that the orthogonality relations
\begin{align*}
	\mathfrak{N}\left[x^k \dfrac{\mathcal{P}_n(x)}{\prod_{j=1}^{n}(x-a_j)(x-b_j)} \right]\neq  0, \quad 0\leq k < n,
\end{align*}
hold \cite[Theorem 3.5]{Esmail masson JAT 1995}. The recurrence relation \eqref{R2 Recurrence Relation} shall be referred to as the recurrence relation of $R_{II}$ type and the polynomials $\mathcal{P}_n(x)$,  $n\geq 1$, generated by it as $R_{II}$ polynomials, according to \cite{Esmail masson JAT 1995}.
Note that the infinite continued fraction
\begin{align}\label{R2 continued fraction}
\mathcal{R}_{II}(x) = \frac{1}{\rho_0(x-c_0)} \mathbin{\genfrac{}{}{0pt}{}{}{-}} \frac{d_1(x-a_1)(x-b_1)}{\rho_1(x-c_1)} \mathbin{\genfrac{}{}{0pt}{}{}{-}} \frac{d_2(x-a_2)(x-b_2)}{\rho_2(x-c_2)} \mathbin{\genfrac{}{}{0pt}{}{}{-}} \mathbin{\genfrac{}{}{0pt}{}{}{\cdots}} ,
\end{align}
terminates when $x=a_k$ or $x=b_k$, $k \geq 1$. A method for constructing approximation of continued fractions is developed recently in \cite{Chen Srivastava Wang RACSAM 2021}.
As a natural consequence of the aforementioned studies, the purpose of this work is to investigate the self perturbation of $R_{II}$ polynomials and derive relationships between parameters so that the resulting sequence becomes a sequence of $R_{II}$ polynomials again. We refer to \cite{KKB Swami PAMS 2019, KKB LAA 2021, Bracciali Pereira ranga 2020} and references therein for some recent developments in the theory of $R_{II}$ polynomials.  In \cite{swami vinay R2 2022}, a specific form of $R_{II}$ recurrence
\begin{align}\label{special R2 Linear pp}
	&\mathcal{P}_{n+1}(x) =\rho_n (x-c_{n})\mathcal{P}_n(x)-d_{n} (x^2+\omega^2)\mathcal{P}_{n-1}(x), \quad n\geq 1, \\
	& \nonumber	\mathcal{P}_{0}(x) = 1, \qquad \mathcal{P}_1(x) = x-c_0,
\end{align}
was considered, where $\{ \rho_n \geq 1\}_{n \geq 0}$ and $\{c_n\}_{n \geq 0}$ are real sequences and $\{d_{n}\}_{n \geq 1}$ is a positive chain sequence (see \cite{Chihara book 1978}). This recurrence relation is linked to a generalized eigenvalue problem whose eigenvalues are zeros of $\mathcal{P}_n(x)$ in this work. \par
%A generalization of complementary Routh-Romanovski polynomials, studied in \cite{Finkelshtein Ribeiro Ranga Tyaglov PAMS 2019}, is introduced. It is, then, used to prove the existence of a non-constant unique sequence $\{\alpha_{n}\}_{n=0}^\infty$.\par
Two polynomial sequences $\{\mathcal{U}_n(x)\}_{n=0}^\infty$ and $\{\mathcal{V}_n(x)\}_{n=0}^\infty$ are said to be biorthogonal with regard to a moment functional $\mathcal{M}$ if
\begin{align*}
	\mathcal{M}(\mathcal{U}_n(x)\mathcal{V}_m(x))=\kappa_n\delta_{n,m}, \quad \kappa_n \neq 0, \quad n,m \geq 0,
\end{align*}
holds \cite{Konhauser JAMA 1965}. It's worth noting that, unlike conventional orthogonality, biorthogonality uses two separate sequences. The biorthogonality employed in this work is between two rational functions $\psi_i(\lambda)$ and $\eta_k(\lambda)$ with respect to a finite discrete measure located at the points $\lambda_s$ with the weights $\Omega_s$ and is defined by the condition \cite{Zhedanov biorthogonal JAT 1999}
\begin{align*}
	\sum_{s=0}^{N}\Omega_s\psi_i(\lambda_s)\eta_k(\lambda_s)=\delta_{i,k}, \quad i,k=0,1,\ldots,N.
\end{align*}
Biorthogonal polynomials also appear while studying the two-matrix model in the theory of random matrices and have the property that
\begin{align*}
\int_{\mathbb{R}}\int_{\mathbb{R}}\mathcal{U}_n(x)\mathcal{V}_m(y)e^{-p(x)-q(y)+2\tau xy}dx dy = 0, \quad n \neq m,
\end{align*}
where $p(x)$ and $q(y)$ are polynomial potentials sufficiently large at infinity and $\tau$ is a non-zero constant \cite{Kuijlaars McLaughlin 2005}. The biorthogonal polynomials are characterization of multiple orthogonal polynomials (MOPs). MOPs are useful in solving Hermite-Pad\'e approximation problems related to Nikishin system. The self perturbed sequence $\{\mathcal{L}_n(x)\}_{n=0}^\infty$ is used to derive biorthogonality relations between rational (or eigen) functions generated from the generalised eigenvalue problem. \par

The interlacing properties of polynomial zeros are crucial in numerical quadrature, where zeros are used as nodes \cite{Bracciali Pereira ranga 2020}, approximation theory and various other applications \cite{Beardon Driver JAT 2005, Brezinski Driver Redivo-Zaglia ANM 2004, Jordaan Tookos 2009,Joulak ANM 2005}. With this in mind, a thorough investigation of the interlacing properties of zeros of $R_{II}$ polynomials and perturbed polynomials has been carried out. Another significant question in the study of such interlacing properties is whether and when the zeros of two different sequences of $R_{II}$ polynomials separate each other. From a theoretical standpoint, a solution to this is supplied by generating two separate self perturbed sequences, supported by an example. \par

The structural organisation of the paper is as follows: The generalized complementary Romanovski-Routh polynomials (GCRR polynomials), which are new in the literature, are constructed in this manuscript. In \Cref{Recurrence Relations of R2 type}, a three-term recurrence relation for self perturbed GCRR polynomials is generated, which is further reduced to a special $R_{II}$-type recurrence \eqref{special R2 Linear pp} under specific requirements on recurrence coefficients. In \Cref{Biorthogonlality from linear combination}, the generalised eigenvalue representation $x\mathcal{J}_n-\mathcal{K}_n$ arising from $\{\mathcal{L}_n(x)\}_{n=0}^\infty$ is obtained. Biorthogonality relations are also discussed using Cholesky decomposition, Darboux transformation and $\mathcal{L}\mathcal{D}\mathcal{U}$ decomposition of $\mathcal{J}_n$. \Cref{Interlacing Properties} is dedicated to the interlacing properties of zeros. Further, the existence of a unique sequence $\{\alpha_{n}\}_{n=0}^\infty$ is demonstrated so that the resultant of self perturbation of GCRR polynomials is useful in recovering another $R_{II}$ polynomial. An illustration to justify the developed concepts is provided.

%The connection between co-recursive $R_{II}$ polynomials and self perturbed polynomials is discussed.

\section{Generalized CRR polynomials and their self perturbation}\label{Recurrence Relations of R2 type}
The present section serves the purpose of defining the sequence $\{\alpha_{n}\}_{n=0}^\infty$ with the requirement that $\{\mathcal{L}_n(x)=\mathcal{P}_n(x)- \alpha_n\mathcal{P}_{n-1}(x)\}_{n=0}^\infty$ becomes a sequence of $R_{II}$ polynomials, where $\mathcal{P}_n(x)$ are the GCRR polynomials. For the complementary Romanovski-Routh polynomials studied in \cite{Finkelshtein Ribeiro Ranga Tyaglov PAMS 2019, Finkelshtein Ribeiro Ranga Tyaglov CRR 2020}, a generalization can be constructed in the following way:

The Jacobi polynomials corresponding to the weight function $\mu^{(\alpha, \beta)}(x)=(\omega-x)^{\beta-1+\frac{i\alpha}{2}}(\omega+x)^{\beta-1-\frac{i\alpha}{2}}$ are given \cite{Koekoek Lesky Swarttouw book 2010} as
\begin{align*}
\mathcal{H}^{(\beta-1+\frac{i\alpha}{2},\beta-1-\frac{i\alpha}{2})}_{n}(x)=\dfrac{(2\omega)^n(\beta+\frac{i\alpha}{2})_n}{(n+2\beta-1)_n}{}_2 F_1\left(-n,n+2\beta-1; \beta+\frac{i\alpha}{2};\dfrac{\omega-x}{2\omega}\right).
\end{align*}
When $x$ is purely imaginary, the modified version and related Rodrigues formula \cite{Koekoek Lesky Swarttouw book 2010} can be written, respectively, as
\begin{align*}
&\mathcal{H}^{(\beta-1+\frac{i\alpha}{2},\beta-1-\frac{i\alpha}{2})}_{n}(ix)=\dfrac{(2\omega)^n(\beta+\frac{i\alpha}{2})_n}{(n+2\beta-1)_n} {}_2 F_1\left(-n,n+2\beta-1; \beta+\frac{i\alpha}{2};\dfrac{\omega-ix}{2\omega}\right), \\	&\mathcal{H}^{(\beta-1+\frac{i\alpha}{2},\beta-1-\frac{i\alpha}{2})}_{n}(ix)=\dfrac{(n+2\beta-1)_n^{-1}(-1)^n(-i)^n}{(\omega^2+x^2)^{\eta-1}e^{-\cot^{-1}\frac{x}{w}}}\dfrac{d^n}{dx^n}[(\omega^2+x^2)^{n+\eta-1}e^{-\cot^{-1}\frac{x}{w}}],
\end{align*}
and the weight function $\mu^{(\alpha, \beta)}(x)$ becomes $\mu^{(\alpha, \beta)}(ix)=(\omega^2+x^2)^{\eta-1}e^{-\cot^{-1}\frac{x}{w}}$. Then, a generalization of Rodrigues formula for Romanovski-Routh polynomials is given by
\begin{align*}
\mathcal{R}^{(\alpha,\beta)}_{n}(x)=\dfrac{1}{\mu^{(\alpha, \beta)}(ix)} \dfrac{d^n}{dx^n}[\mu^{(\alpha, \beta)}(ix)(\omega^2+x^2)^{n}].
\end{align*}
The polynomials $\mathcal{R}^{(\alpha,\beta)}_{n}(x)$ and $\mathcal{H}^{(\beta-1+\frac{i\alpha}{2},\beta-1-\frac{i\alpha}{2})}_{n}(ix)$ are related by the expression
\begin{align*}	\mathcal{R}^{(\alpha,\beta)}_{n}(x)=(-i)^n(n+2\beta-1)_n\mathcal{H}^{(\beta-1+\frac{i\alpha}{2},\beta-1-\frac{i\alpha}{2})}_{n}(ix).
\end{align*}
It is easy to verify that $\mathcal{R}^{(\alpha,\beta)}_{n}(x)$ are the solutions of the differential equation
\begin{align*}
(x^2+\omega^2)y''(x)+(2\beta x+\alpha \omega)y'(x)-n(n+2\beta-1)y(x)=0,
\end{align*}
and satisfy a finite orthogonality condition
\begin{align*}
\int_{-\infty}^{\infty}\mathcal{R}^{(\alpha,\beta)}_{n}(x)\mathcal{R}^{(\alpha,\beta)}_{m}(x)\mu^{(\alpha, \beta)}(ix) dx = 0, \quad m \neq n,
\end{align*}
when $m+n-1<-2\beta$ and $\beta$ is large enough negative number.\par
The polynomials complementary to $\mathcal{Q}^{(\alpha,\beta)}_{n}(x)$ are defined using a variation of Rodrigues formula given in \cite{Weber 2007} as
\begin{align*}
\mathcal{Q}^{(\alpha,\beta)}_{n}(x)=\dfrac{(x^2+\omega^2)^n}{\mu^{(\alpha, \beta)}(ix)} \dfrac{d^n}{dx^n}\mu^{(\alpha, \beta)}(ix).
\end{align*}
It is established in \cite{Raposo Weber Castillo Kirchbach 2007, Weber 2007} that $\mathcal{Q}^{(\alpha,\beta)}_{n}(x)=\mathcal{R}^{(\alpha,\beta-n)}_{n}(x)$, therefore,
\begin{align}\label{Q_n alpha beta}
\mathcal{Q}^{(\alpha,\beta)}_{n}(x)=(-2i\omega)^n(\beta-n+\frac{i\alpha}{2})_n{}_2 F_1\left(-n,2\beta-n-1; \beta-n+ \frac{i\alpha}{2}; \dfrac{\omega-ix}{2\omega}\right).
\end{align}
The generalized complementary Romanovski-Routh polynomials (GCRR) can now be defined as
\begin{align}\label{P_n to Q_n connection}
\mathcal{P}_{n}(x):=\dfrac{(-1)^n}{2^n(\zeta)_n}\mathcal{Q}^{(2\theta,-\zeta+1)}_{n}(x), \quad n \geq 1,
\end{align}
and are given by the hypergeometric expression
\begin{align} \label{GCRR polynomial}
\mathcal{P}_{n}(x) = \dfrac{(x-i\omega)^n}{2^n}\dfrac{(2\zeta)_n}{(\zeta)_n}{}_2 F_1 \left(-n,e; e+\bar{e};\dfrac{-2i}{x-i\omega}\right),
\end{align}
where $e = \zeta+i\theta$, $\zeta > 0$. This hypergeometric expression is obtained from \eqref{Q_n alpha beta} using two Phaff transformations given in \cite{Andrews Askey Roy book 1999} (see eq. (2.2.6) and (2.3.14)). They are orthogonal with respect to the weight
\begin{align*}
\omega^{(\zeta,\theta)}(x)= \dfrac{2^{2\zeta-1}|\Gamma(e)|^2 e^{\theta \pi} (e^{-\cot^{-1}(\frac{x}{\omega})})^{2\theta}}{\Gamma(2\zeta-1)2\pi (\omega^2+x^2)^\zeta},
\end{align*}
and satisfy the recurrence relation
\begin{align}\label{GCRR recurrence}
\mathcal{P}_{n+1}(x)= \left(x-\dfrac{\theta}{\zeta+n}\right)\mathcal{P}_n(x)-\dfrac{n(2\zeta+n-1)}{4(\zeta+n)(\zeta+n-1)} (x^2+\omega^2)\mathcal{P}_{n-1}(x), \quad  n \geq 1,
\end{align}
where $\mathcal{P}_{0}(x)=1$ and $\mathcal{P}_{1}(x)=x-c_1$.

In \cite{Draux Integral transform 2016}, it is proved that quasi-orthogonal polynomials satisfy recurrence relations with polynomial coefficients. The polynomials obtained on self perturbation of $R_{I}$ polynomials \cite{KKB Swami 2018} are shown to satisfy a three-term recurrence relation with polynomial coefficients of degree at most two. The first result shows that a three-term recurrence relation with polynomial coefficients of degree at most four exists when GCRR polynomials are self perturbed. \par

\begin{theorem}\label{Theorem Linear R2 general RR}
%With assumptions of \Cref{Theorem Linear R2 general RR},
Let	$\{\mathcal{P}_n(x)\}_{n=0}^\infty$ be a sequence of GCRR polynomials generated by \eqref{GCRR recurrence}.
%\begin{align}\label{Linear combination Q_n}
%	\mathcal{S}_n(x)=\mathcal{P}_n(x)+\alpha_n\mathcal{P}_{n-1}(x), \quad \alpha_n \in \mathbb{R}\backslash \{ 0 \}, \quad n\geq 0,
%\end{align}
Then the sequence of self perturbed polynomials	$\{\mathcal{L}_n(x)\}_{n=0}^\infty$, defined by \eqref{Linear combination Q_n}, satisfies the three term recurrence relation with polynomial coefficients of the form
\begin{align}\label{Linear R2 general RR}
(e_nx^2+f_nx+g_n)\mathcal{L}_{n+1}(x)&=(p_nx^3+q_nx^2+r_nx+s_n)\mathcal{L}_n(x) \nonumber \\ &+(t_nx^4+u_nx^3+v_nx^2+w_nx+z_n)\mathcal{L}_{n-1}(x), \quad n\geq 1,
\end{align}
with initial conditions $\mathcal{L}_1(x)=\rho_0(x-\alpha_1\rho_0^{-1}-c_0)$ and $\mathcal{L}_0(x)=0$. The constants $e_n,~ f_n,~ g_n,~ p_n,~ q_n,~ r_n,~ s_n,~ t_n,~ u_n,~ v_n,~ w_n,~ z_n$, $n\geq 1$, are given by
\begin{align*}
&e_n= d_{n-1}, \quad f_n=-\alpha_{n-1}\rho_{n-1}, \quad g_n=d_{n-1}\omega^2+\alpha_{n-1}(\alpha_{n}+\rho_{n-1}c_{n-1}), \\
& p_n=\rho_ne_{n}, \quad q_n=\alpha_{n-1}(d_{n}-\rho_n\rho_{n-1})-d_{n-1}[\rho_n c_n+\alpha_{n+1}],\\
&r_n=\alpha_{n-1}\rho_{n-1}[\alpha_{n+1}+\rho_n(c_n+c_{n-1})]
+\rho_n d_{n-1}\omega^2, \quad t_n=-e_{n+1}d_{n-1}, \\ &s_n=-(\alpha_{n-1}\rho_{n-1}c_{n-1}+d_{n-1}\omega^2)(\rho_nc_n+\alpha_{n+1})+\alpha_{n-1}d_{n}\omega^2, \\
&u_n\omega^2=-d_{n-1}f_{n+1}\omega^2=w_n, \quad v_n =-d_{n-1}[g_{n+1}+\omega^2e_{n+1}], \quad z_n=-d_{n-1}\omega^2g_{n+1}, \quad n\geq 1.
\end{align*}	
\end{theorem}

\begin{proof}
To find a relation between $\mathcal{L}_{n+1}(x)$, $\mathcal{L}_{n}(x)$ and $\mathcal{L}_{n-1}(x)$, we eliminate $\mathcal{P}_{n+1}(x)$, $\mathcal{P}_{n}(x)$, $\mathcal{P}_{n-1}(x)$ and $\mathcal{P}_{n-2}(x)$ from the following set of equations
\begin{align}
\mathcal{L}_{j+1}(x)&=\mathcal{P}_{j+1}(x)-\alpha_{j+1}\mathcal{P}_{j}(x), \quad j=n,n-1,n-2, \label{five eq combined 1}  \\
%\mathcal{S}_n(x)&=\mathcal{P}_n(x)-\alpha_n\mathcal{P}_{n-1}(x), \label{five eq combined 2} \\
%\mathcal{S}_{n-1}(x)&=\mathcal{P}_{n-1}(x)-\alpha_{n-1}\mathcal{P}_{n-2}(x), \label{five eq combined 3} \\
\mathcal{P}_{j+1}(x) &= \rho_j(x-c_j)\mathcal{P}_j(x)-d_j (x^2+\omega^2)\mathcal{P}_{j-1}(x), \quad j=n,n-1. \label{five eq combined 4}
%\mathcal{P}_{n}(x) &= \rho_{n-1}(x-c_{n-1})\mathcal{P}_{n-1}(x)-d_{n-1} (x^2+\omega^2)\mathcal{P}_{n-2}(x). \label{five eq combined 5}
\end{align}
For $j=n$, multiplying \eqref{five eq combined 1}  by $d_{n-1}(x^2+\omega^2)-\alpha_{n-1}\rho_{n-1}(x-c_{n-1})$ and deploying \eqref{five eq combined 4}, we have
\newline
$\displaystyle
[d_{n-1}(x^2+\omega^2)-\alpha_{n-1}\rho_{n-1}(x-c_{n-1})]\mathcal{L}_{n+1}(x)
$
\begin{align*}
&=[(\rho_n(x-c_n)-\alpha_{n+1})[ d_{n-1}(x^2+\omega^2)-\alpha_{n-1} \rho_{n-1} (x-c_{n-1})]] \mathcal{P}_{n}(x) \\
&-d_n (x^2+\omega^2)[ d_{n-1}(x^2+\omega^2)-\alpha_{n-1} \rho_{n-1} (x-c_{n-1})] \mathcal{P}_{n-1}(x).
\end{align*}
%Substituting \eqref{five eq combined 1}, with $j=n-1$, on the right side of the above equality gives
%\begin{align*}
%&[d_{n-1}(x^2+\omega^2)-\alpha_{n-1}\rho_{n-1}(x-c_{n-1})]\mathcal{L}_{n+1}(x)\\
%&=[(\rho_n(x-c_n)-\alpha_{n+1})[ d_{n-1}(x^2+\omega^2)-\alpha_{n-1} \rho_{n-1} (x-c_{n-1}) ]][\mathcal{L}_{n}(x)+\alpha_n\mathcal{P}_{n-1}(x)] \\
%&-d_n (x^2+\omega^2)[ d_{n-1}(x^2+\omega^2)-\alpha_{n-1} \rho_{n-1} (x-c_{n-1})] \mathcal{P}_{n-1}(x) \\
%&=[(\rho_n(x-c_n)-\alpha_{n+1})\{ d_{n-1}(x-a_{n-1})(x-b_{n-1})-\alpha_{n-1} \rho_{n-1} (x-c_{n-1}) \}]\mathcal{S}_{n}(x) \\
%&+\alpha_n[(\rho_n(x-c_n)-\alpha_{n+1})\{ d_{n-1}(x-a_{n-1})(x-b_{n-1})-\alpha_{n-1} \rho_{n-1} (x-c_{n-1}) \}]\mathcal{P}_{n-1}(x) \\
%&-d_n (x-a_n)(x-b_n)\{ d_{n-1}(x-a_{n-1})(x-b_{n-1})-\alpha_{n-1} \rho_{n-1} (x-c_{n-1})\} \mathcal{P}_{n-1}(x) \\
%&=[(\rho_n(x-c_n)-\alpha_{n+1})[ d_{n-1}(x^2+\omega^2)-\alpha_{n-1} \rho_{n-1} (x-c_{n-1}) ]]\mathcal{L}_{n}(x) \\
%&+[(\alpha_n\rho_n(x-c_n)-\alpha_n \alpha_{n+1}-d_n (x^2+\omega^2))[ d_{n-1}(x^2+\omega^2)-\alpha_{n-1} \rho_{n-1} (x-c_{n-1}) ]]\mathcal{P}_{n-1}(x).
%	&
%&=[(\rho_n(x-c_n)-\alpha_{n+1})\{ d_{n-1}(x-a_{n-1})(x-b_{n-1})-\alpha_{n-1} \rho_{n-1} (x-c_{n-1}) \}]\mathcal{S}_{n}(x) \\
%&+d_{n-1}(x-a_{n-1})(x-b_{n-1})(\alpha_n\rho_n(x-c_n)-\alpha_n \alpha_{n+1}-d_n (x-a_n)(x-b_n)) \mathcal{P}_{n-1}(x)\\
%&-\alpha_{n-1} \rho_{n-1} (x-c_{n-1})(\alpha_n\rho_n(x-c_n)-\alpha_n \alpha_{n+1}-d_n (x-a_n)(x-b_n)) \mathcal{P}_{n-1}(x).
%\end{align*}
Substituting \eqref{five eq combined 1} with $j=n-1$ and $j=n-2$ on the right side of the above equality gives
\newline
$\displaystyle
[d_{n-1}(x^2+\omega^2)-\alpha_{n-1}\rho_{n-1}(x-c_{n-1})]\mathcal{L}_{n+1}(x)
$
\begin{align*}
%&=[(\rho_n(x-c_n)-\alpha_{n+1})[ d_{n-1}(x^2+\omega^2)-\alpha_{n-1} \rho_{n-1} (x-c_{n-1}) ]]\mathcal{L}_{n}(x) \\
%&+d_{n-1}(x^2+\omega^2)(\alpha_n\rho_n(x-c_n)-\alpha_n \alpha_{n+1}-d_n (x^2+\omega^2)) [\mathcal{L}_{n-1}(x)+\alpha_{n-1}\mathcal{P}_{n-2}(x)]\\
%&-\alpha_{n-1} \rho_{n-1} (x-c_{n-1})(\alpha_n\rho_n(x-c_n)-\alpha_n \alpha_{n+1}-d_n (x^2+\omega^2)) \mathcal{P}_{n-1}(x) \\
%&=[(\rho_n(x-c_n)-\alpha_{n+1})\{ d_{n-1}(x-a_{n-1})(x-b_{n-1})-\alpha_{n-1} \rho_{n-1} (x-c_{n-1}) \}]\mathcal{S}_{n}(x) \\
%&+d_{n-1}(x-a_{n-1})(x-b_{n-1})(\alpha_n\rho_n(x-c_n)-\alpha_n \alpha_{n+1}-d_n (x-a_n)(x-b_n)) \mathcal{S}_{n-1}(x)\\
%&+\alpha_{n-1}d_{n-1}(x-a_{n-1})(x-b_{n-1})(\alpha_n\rho_n(x-c_n)-\alpha_n \alpha_{n+1}-d_n (x-a_n)(x-b_n))\mathcal{P}_{n-2}(x) \\
%&-\alpha_{n-1} \rho_{n-1} (x-c_{n-1})(\alpha_n\rho_n(x-c_n)-\alpha_n \alpha_{n+1}-d_n (x-a_n)(x-b_n)) \mathcal{P}_{n-1}(x) \\
&=[(\rho_n(x-c_n)-\alpha_{n+1})[ d_{n-1}(x^2+\omega^2)-\alpha_{n-1} \rho_{n-1} (x-c_{n-1}) ]]\mathcal{L}_{n}(x) \\
&+d_{n-1}(x^2+\omega^2)(\alpha_n\rho_n(x-c_n)-\alpha_n \alpha_{n+1}-d_n (x^2+\omega^2)) \mathcal{L}_{n-1}(x)\\
&-\alpha_{n-1}(\alpha_n\rho_n(x-c_n)-\alpha_n \alpha_{n+1}-d_n (x^2+\omega^2))[\rho_{n-1} (x-c_{n-1})\mathcal{P}_{n-1}(x)\\
&-d_{n-1}(x^2+\omega^2)\mathcal{P}_{n-2}(x)].
\end{align*}
Now, substituting \eqref{five eq combined 4} with $j=n-1$, in the last expression implies
\newline
$\displaystyle
[d_{n-1}(x^2+\omega^2)-\alpha_{n-1}\rho_{n-1}(x-c_{n-1})]\mathcal{L}_{n+1}(x)
$
\begin{align*}
&=[(\rho_n(x-c_n)-\alpha_{n+1})[ d_{n-1}(x^2+\omega^2)-\alpha_{n-1} \rho_{n-1} (x-c_{n-1}) ]]\mathcal{L}_{n}(x) \\
&+d_{n-1}(x^2+\omega^2)(\alpha_n\rho_n(x-c_n)-\alpha_n \alpha_{n+1}-d_n (x^2+\omega^2)) \mathcal{L}_{n-1}(x)\\
&-\alpha_{n-1}(\alpha_n\rho_n(x-c_n)-\alpha_n \alpha_{n+1}-d_n (x^2+\omega^2))\mathcal{P}_{n}(x).
\end{align*}
Adding and substracting $\alpha_{n-1}d_n (x^2+\omega^2)\mathcal{L}_{n}(x)$, we get
\newline
$\displaystyle
[d_{n-1}(x^2+\omega^2)-\alpha_{n-1}\rho_{n-1}(x-c_{n-1})]\mathcal{L}_{n+1}(x)
$
\begin{align*}
&=[(\rho_n(x-c_n)-\alpha_{n+1})[ d_{n-1}(x^2+\omega^2)-\alpha_{n-1} \rho_{n-1} (x-c_{n-1})+\alpha_{n-1}d_n (x^2+\omega^2) ]] \mathcal{L}_{n}(x)\\
&+d_{n-1}(x^2+\omega^2)(\alpha_n\rho_n(x-c_n)-\alpha_n \alpha_{n+1}-d_n (x^2+\omega^2)) \mathcal{L}_{n-1}(x)\\
&-\alpha_{n-1}(\alpha_n\rho_n(x-c_n)-\alpha_n \alpha_{n+1}-d_n (x^2+\omega^2))\mathcal{P}_{n}(x)\\
&-\alpha_{n-1}d_n (x^2+\omega^2)\mathcal{P}_{n}(x)+\alpha_{n-1}\alpha_{n}d_n (x^2+\omega^2)\mathcal{P}_{n-1}(x).
%&=[(\rho_n(x-c_n)-\alpha_{n+1})\{ d_{n-1}(x-a_{n-1})(x-b_{n-1})-\alpha_{n-1} \rho_{n-1} (x-c_{n-1})+\alpha_{n-1}d_n (x-a_n)(x-b_n) \}] \\
%&\mathcal{S}_{n}(x)+d_{n-1}(x-a_{n-1})(x-b_{n-1})(\alpha_n\rho_n(x-c_n)-\alpha_n \alpha_{n+1}-d_n (x-a_n)(x-b_n)) \mathcal{S}_{n-1}(x)\\
%&+\alpha_{n-1}\alpha_{n}\alpha_{n+1}\mathcal{P}_{n}(x)-\alpha_{n-1}\alpha_{n}[\rho_n(x-c_n)\mathcal{P}_{n}(x)-d_n (x-a_n)(x-b_n)\mathcal{P}_{n-1}(x)].
\end{align*}
%By virtue of \eqref{five eq combined 4} with $j=n$, we obtain
%\begin{align*}
%&[d_{n-1}(x^2+\omega^2)-\alpha_{n-1}\rho_{n-1}(x-c_{n-1})]\mathcal{L}_{n+1}(x)= \\
%&=[(\rho_n(x-c_n)-\alpha_{n+1})[ d_{n-1}(x^2+\omega^2)-\alpha_{n-1} \rho_{n-1} (x-c_{n-1})+\alpha_{n-1}d_n (x^2+\omega^2)]] \mathcal{L}_{n}(x)\\
%&+d_{n-1}(x^2+\omega^2)(\alpha_n\rho_n(x-c_n)-\alpha_n \alpha_{n+1}-d_n (x^2+\omega^2)) \mathcal{L}_{n-1}(x)\\
%&-\alpha_{n-1}\alpha_{n}[\mathcal{P}_{n+1}(x)-\alpha_{n+1}\mathcal{P}_{n}(x)].
%\end{align*}
Again using \eqref{five eq combined 4} and \eqref{five eq combined 1} with $j=n$, we end up getting
\newline
$\displaystyle
[d_{n-1}(x^2+\omega^2)-\alpha_{n-1}\rho_{n-1}(x-c_{n-1})+\alpha_{n-1}\alpha_{n}]\mathcal{L}_{n+1}(x)
$
\begin{align*}
&=[(\rho_n(x-c_n)-\alpha_{n+1})[ d_{n-1}(x^2+\omega^2)-\alpha_{n-1} \rho_{n-1} (x-c_{n-1})+\alpha_{n-1}d_n (x^2+\omega^2) ]]\mathcal{L}_{n}(x) \\
&+d_{n-1}(x^2+\omega^2)(\alpha_n\rho_n(x-c_n)-\alpha_n \alpha_{n+1}-d_n (x^2+\omega^2)) \mathcal{L}_{n-1}(x).
\end{align*}
The constants can be computed simply by collecting coefficients of the monomials.
\end{proof}

The following proposition is a generalisation of \Cref{Theorem Linear R2 general RR} for the general $R_{II}$ type recurrence \eqref{R2 Recurrence Relation}. As the proof is similar to the one for \Cref{Theorem Linear R2 general RR}, it has been omitted.
%With $a_n=i\omega$ and $b_n=-i\omega$, $n \geq 0$, used to obtain special ,  can be proposed:

\begin{proposition}\label{proposition simplified constants}
The sequence of self perturbed $R_{II}$ polynomials defined by \eqref{R2 Recurrence Relation} satisfy the three term recurrence relation \eqref{Linear R2 general RR} with the generalised constants
\begin{align*}
&e_n= d_{n-1}, \quad f_n=-d_{n-1}(a_{n-1}+b_{n-1})-\alpha_{n-1}\rho_{n-1}, \\
&g_n=d_{n-1}a_{n-1}b_{n-1}+\alpha_{n-1}(\alpha_{n}+\rho_{n-1}c_{n-1}),\quad p_n=\rho_nd_{n-1},\\
&q_n=\alpha_{n-1}(d_{n}-\rho_n\rho_{n-1})-d_{n-1}[\rho_n(a_{n-1}+b_{n-1}+c_n)+\alpha_{n+1}],\\
&r_n=\alpha_{n-1}\rho_{n-1}[\alpha_{n+1}+\rho_n(c_n+c_{n-1})]+d_{n-1}(a_{n-1}+b_{n-1})(\alpha_{n+1}+\rho_nc_n)\\
&+\alpha_{n-1}d_{n}(a_{n}+b_{n})+\rho_n d_{n-1}a_{n-1}b_{n-1}, \\ &s_n=-(\alpha_{n-1}\rho_{n-1}c_{n-1}+d_{n-1}a_{n-1}b_{n-1})(\rho_nc_n+\alpha_{n+1})+\alpha_{n-1}d_{n}a_{n}b_{n},\\
& u_n=d_{n-1}[d_{n}(a_n+b_n+a_{n-1}+b_{n-1})+\alpha_{n}\rho_{n}], \quad t_n=-d_{n}d_{n-1}, \\
&v_n =-d_{n-1}[\alpha_{n}( \alpha_{n+1}+\rho_{n}c_{n})+d_{n}(a_{n}b_{n}+a_{n-1}b_{n-1})+(a_{n-1}+b_{n-1})(d_n(a_{n}+b_{n})+\alpha_{n}\rho_n)],\\
& w_n=-d_{n-1}[a_{n-1}b_{n-1}(-d_n(a_{n}+b_{n})-\alpha_{n}\rho_n)-(a_{n-1}+b_{n-1})(\alpha_{n}( \alpha_{n+1}+\rho_{n}c_{n})+d_na_nb_n)], \\
&z_n=-d_{n-1}[a_{n-1}b_{n-1}(\alpha_{n}( \alpha_{n+1}+\rho_{n}c_{n})+d_na_nb_n)].
\end{align*}
\end{proposition}

\begin{remark}
Since $\mathcal{L}_0(x)=\mathcal{P}_0(x)-\alpha_0\mathcal{P}_{-1}(x)$ and $\mathcal{P}_{-1}(x)=0$, the choice of $\alpha_0$ in the sequence $\{\alpha_n\}_{n=0}^\infty$ is redundant.
\end{remark}

We will simplify \eqref{Linear R2 general RR} by imposing some conditions on the constants derived in \Cref{Theorem Linear R2 general RR}, which will be helpful in further discussion. The following is an immediate consequence:

\begin{theorem}\label{Theorem 1 with aplha_n condition}
Let the sequence $\{\alpha_n\}_{n=1}^\infty$ and the paramaters obtained on comparing \eqref{GCRR recurrence} with \eqref{special R2 Linear pp} be such that they satisfy following conditions
\begin{align}
&\alpha_n=\rho_{n-1}(1-c_{n-1})-\omega^2\alpha^{-1}_{n-1}d_{n-1}, \label{condition 1}\\
&d_n\alpha_{n-1}\rho_{n-1}=d_{n-1}\alpha_{n}\rho_{n}, \quad n\geq 1, \label{condition 2}
\end{align}
then the three term recurrence relation for $\mathcal{L}_n(x)$ \eqref{Linear R2 general RR} reduces to a special $R_{II}$ type recurrence given by
\begin{align}\label{R2 recurrence from linear 1}
\mathcal{L}_{n+1}(x)=\rho_n\left(x+\dfrac{f_n^{-1}s_n}{\rho_n}\right)\mathcal{L}_n(x)-d_n(x^2+\omega^2)\mathcal{L}_{n-1}(x), \quad n\geq 1,
\end{align}
with $\mathcal{L}_1(x)=\rho_0(x-\alpha_1\rho_0^{-1}-c_0)$ and $\mathcal{L}_0(x)=0$.
\end{theorem}

\begin{proof}
Relation \eqref{condition 1} implies $f_n=-g_n$ and $\omega^2 u_n=w_n=-z_n=-v_n\omega^2+t_n\omega^4$ from which \eqref{Linear R2 general RR} gives
\begin{align}\label{TTRR 1 in main theorem}
(e_nx^2+&f_nx-f_n)\mathcal{L}_{n+1}(x)=(p_nx^3+q_nx^2+r_nx+s_n)\mathcal{L}_n(x) \nonumber \\ &+(t_nx^4+u_nx^3+(-u_n+t_n\omega^2)x^2+\omega^2u_nx-u_n\omega^2)\mathcal{L}_{n-1}(x), \quad n\geq 1.
\end{align}
Clearly, $(x^2+\omega^2)$ is a factor of the coefficient of $\mathcal{S}_{n-1}(x)$ in \eqref{TTRR 1 in main theorem}. Thus, the former implies
\begin{align}\label{TTRR 2 in main theorem}
e_n\left(x^2+\dfrac{f_n}{e_n}x-\dfrac{f_n}{e_n}\right)&\mathcal{L}_{n+1}(x)=(p_nx^3+q_nx^2+r_nx+s_n)\mathcal{L}_n(x) \nonumber \\ &+t_n(x^2+\omega^2)\left(x^2+\dfrac{u_n}{t_n}x-\dfrac{u_n}{t_n}\right)\mathcal{L}_{n-1}(x), \quad n\geq 1.
\end{align}
Suppose that $\delta$ is a root of $x^2+\dfrac{f_n}{e_n}x-\dfrac{f_n}{e_n}$, i.e., $\delta^2+\dfrac{f_n}{e_n}\delta-\dfrac{f_n}{e_n}=0$. Then, using the identity $\delta^2=\dfrac{f_n}{e_n}(\delta-1)$ and \eqref{condition 2}, one can prove that $p_n\delta^3+q_n\delta^2+r_n\delta+s_n=0$, i.e. $\delta$ is also a root of $p_nx^3+q_nx^2+r_nx+s_n$. Thus, the two factors of the polynomial coefficient of $\mathcal{L}_{n+1}(x)$ are also the factors of the polynomial coefficient of $\mathcal{L}_{n}(x)$. Let the third root of the coefficient of $\mathcal{S}_{n}(x)$ be $\gamma$, then we have $\gamma = -\dfrac{s_ne_n}{p_nf_n}$. These facts, when used along with \eqref{condition 2}, reduces \eqref{TTRR 2 in main theorem} to
\begin{align}
e_n&\mathcal{L}_{n+1}(x)=p_n\left(x+\dfrac{s_ne_n}{p_nf_n}\right)\mathcal{L}_n(x)  +t_n(x^2+\omega^2)\mathcal{L}_{n-1}(x), \quad n\geq 1,
\end{align}
which, on simplifying, gives the desired result.
\end{proof}

\begin{remark}\label{remark another alpha_n or beta_n}
The construction of the sequence $\{\mathcal{L}_{n}(x)\}_{n=1}^\infty$, given by \eqref{R2 recurrence from linear 1}, relies on cosidering $f_n=-g_n$ in \eqref{Linear R2 general RR}. A new sequence $\{\beta_n\}_{n=1}^\infty$ can be defined such that $f_n=g_n$ and \eqref{condition 2} holds for $n \geq 1$. Then, a similar analysis given in \Cref{Theorem 1 with aplha_n condition} leads us to another $R_{II}$ type recurrence
\begin{align}
\mathcal{T}_{n+1}(x)=\rho_n\left(x-\dfrac{f_n^{-1}s_n}{\rho_n}\right)\mathcal{T}_n(x)-d_n(x^2+\omega^2)\mathcal{T}_{n-1}(x), \quad n\geq 1,
\end{align}
with $\mathcal{T}_1(x)=\rho_0(x-\beta_1\rho_0^{-1}-c_0)$ and $\mathcal{T}_0(x)=0$.  This gives a new sequence of $R_{II}$ polynomials, say $\mathcal{T}_n(x)$, $n \geq 2$.
\end{remark}

If we need $\mathcal{L}_n(x)$, $n \geq 0$, to satisfy \eqref{special R2 Linear pp}, it is evident that there can be numerous ways to restrict the recurrence parameters and the sequence $\{\alpha_n\}_{n=1}^\infty$. Another such choice is described in the next theorem.

\begin{theorem}\label{Theorem 2 with aplha_n condition}
Suppose the sequence $\{\alpha_n\}_{n=1}^\infty$ and the paramaters of recurrence relation \eqref{special R2 Linear pp} satisfy \eqref{condition 2} along with the condition
\begin{align}
&\alpha_n=-(\alpha^{-1}_{n-1}d_{n-1}\omega^2+\rho_{n-1}c_{n-1}), \quad n\geq 1. \label{condition 3}
\end{align}
Then the three term recurrence relation for $\mathcal{L}_n(x)$ \eqref{Linear R2 general RR} reduces to a special $R_{II}$ type recurrence given by
\begin{align}\label{main RR in main theorem 2}
\mathcal{L}_{n+1}(x)=\rho_n\left(x+\dfrac{f_n^{-1}r_n}{\rho_n}\right)\mathcal{L}_n(x)-d_n(x^2+\omega^2)\mathcal{L}_{n-1}(x), \quad n\geq 1,
\end{align}
with $\mathcal{L}_1(x)=\rho_0(x-\alpha_1\rho_0^{-1}-c_0)$ and $\mathcal{L}_0(x)=0$.
\end{theorem}

\begin{proof}
It follows from \eqref{condition 3} that $g_n=0$. This further implies $s_n=0$, $z_n=0$ and $v_n=\omega^2t_n$. Thus, \eqref{Linear R2 general RR} becomes
\begin{align}\label{TTRR 3 in main theorem 2}
(e_nx^2+f_nx)\mathcal{L}_{n+1}(x)&=(p_nx^3+q_nx^2+r_nx)\mathcal{L}_n(x) \nonumber \\ &+(t_nx^4+u_nx^3+\omega^2t_nx^2+\omega^2u_nx)\mathcal{L}_{n-1}(x), \quad n\geq 1.
\end{align}
Considering the fact that $x \neq 0$, we have
\begin{align}\label{TTRR 4 in main theorem 2}
e_n\left(x+\dfrac{f_n}{e_n}\right)\mathcal{L}_{n+1}(x)&=(p_nx^2+q_nx+r_n)\mathcal{L}_n(x) \nonumber \\ &+t_n(x^2+\omega^2)\left(x+\dfrac{u_n}{t_n}\right)\mathcal{L}_{n-1}(x), \quad n\geq 1.
\end{align}
It is easy to verify that $x=-\dfrac{f_n}{e_n}$ is a zero of $p_nx^2+q_nx+r_n$. The other zero being $\gamma'=- \dfrac{r_ne_n}{p_nf_n}$. Summarizing these observations in \eqref{TTRR 4 in main theorem 2} together with \eqref{condition 2} gives
\begin{align}
e_n\mathcal{L}_{n+1}(x)&=p_n\left(x+\dfrac{r_ne_n}{p_nf_n}\right)\mathcal{L}_n(x)  +t_n(x^2+\omega^2)\mathcal{L}_{n-1}(x), \quad n\geq 1,
\end{align}
and a direct computation implies \eqref{main RR in main theorem 2}.
\end{proof}

Since the polynomials $\mathcal{P}_n(x)$, $n \geq 0$, are orthogonal with respect to measure $\omega^{(\zeta,\theta)}(x)$, it would be of interest to find the deformation of $\omega^{(\zeta,\theta)}(x)$ that maintains $R_{II}$ structure and with respect to which the sequence of self perturbed polynomials $\mathcal{L}_n(x)$, $n \geq 0$, is orthogonal. The existence of such a transformed measure is guaranteed by \cite[Theorem 3.5]{Esmail masson JAT 1995}.

%\section{ and existence of a non-constant unique sequence $\{\alpha_n\}_{n=1}^\infty$}\label{A non-constat sequence}

\section{Biorthogonlality related to self perturbed GCRR polynomials}\label{Biorthogonlality from linear combination}
The importance of recurrence relation \eqref{special R2 Linear pp} can be realised while studying rational functions and their biorthogonality properties. The polynomial $\mathcal{P}_n(x)$ satisfying \eqref{special R2 Linear pp} is the characteristic polynomial of the matrix pencil $x\mathcal{J}_n-\mathcal{G}_n$, whereas the characteristic polynomial corresponding to the matrix pencil $x\mathcal{J}_n-\mathcal{K}_n$ is the self perturbed polynomial $\mathcal{L}_n(x)$, where $\mathcal{K}_n$, $\mathcal{G}_n$ and $\mathcal{J}_n$ are tridiagonal matrices. Observe that the structure of $\mathcal{J}_n$ is preserved while performing such a perturbation. When entries are constants, the importance of studying the spectra of $\mathcal{J}_n$ is recently highlighted in \cite{Shabrawy Shindy RACSAM 2020}. On the other hand, the positive definiteness of $\mathcal{J}_n$ can be used to find various decompositions of $\mathcal{J}_n$, in particular, the $\mathcal{L}\mathcal{U}$, $\mathcal{U}\mathcal{L}$ and $\mathcal{L}\mathcal{D}\mathcal{U}$ decompositions. Each of these decompositions of $\mathcal{J}_n$ produces a different type of biorthogonal relation (see \eqref{u_L cholesky u_R relation}, \eqref{u_L Darboux u_R relation}, \eqref{u_L LDL u_R relation}). To analyse these biorthogonal relations, we develope an eigenvalue representation corresponding to \eqref{R2 recurrence from linear 1}.

%Consider the sequence $\{\mathcal{L}_n(x)\}_{n=0}^\infty$ defined in \Cref{Theorem 1 with aplha_n condition}. However, a comparative investigation can be carried out with other choices.

\begin{theorem}\label{Lemma 1 Biorthogonality section}
Let $ \mathbf{u}^{\mathcal{R}}(x)=[u_0^{\mathcal{R}}(x), \ldots, u_{n-1}^{\mathcal{R}}(x)]^T $ and $ \mathbf{u}^{\mathcal{L}}(x)=[u_0^{\mathcal{L}}(x), \ldots, u_{n-1}^{\mathcal{L}}(x)] $ be the components of the left and right eigenvectors of the infinite tridiagonal linear pencil $x\mathcal{J}-\mathcal{K}$. With $\mathbf{e}_n$ as the $n^{th}$ column of $n \times n$ identity matrix, the matrix representation of \eqref{R2 recurrence from linear 1} corresponding to the two eigen vectors can be given as
\begin{align}
&\mathcal{K}_n\mathbf{u}^{\mathcal{R}}(x)=x\mathcal{J}_n\mathbf{u}^{\mathcal{R}}(x)+\sqrt{d_{n}}(x-i\omega)u_{n}^{\mathcal{R}}(x)\mathbf{e}_n, \label{u_k_R Relation 3} \\
&\mathbf{u}^{\mathcal{L}}(x)\mathcal{K}_n=x\mathbf{u}^{\mathcal{L}}(x)\mathcal{J}_n+\sqrt{d_{n}}(x+i\omega)u_{n}^{\mathcal{L}}(x)\mathbf{e}_n^{T}, \label{u_k_L Relation 3}
\end{align}
where $\mathcal{K}_n$ and $\mathcal{J}_n$ are $n\times n$ tridiagonal matrices of the form
\begin{align*}
\mathcal{K}_n=&
\begin{pmatrix}
	\rho_0c_0+\alpha_1 & i\omega\sqrt{d_1} & 0 & \ldots & 0 & 0\\
	-i\omega\sqrt{d_1} & -f_{1}^{-1}s_{1} & i\omega\sqrt{d_2} & \ldots & 0 & 0\\
	0 & -i\omega\sqrt{d_2} & -f_{2}^{-1}s_{2} & \ldots & 0 & 0\\
	\vdots & \vdots & \vdots & \ddots & \vdots & \vdots\\
	0 & 0 & 0 & \ldots & -f_{n-2}^{-1}s_{n-2} & i\omega\sqrt{d_{n-1}}\\
	0 & 0 & 0 & \ldots & -i\omega\sqrt{d_{n-1}} & -f_{n-1}^{-1}s_{n-1}
\end{pmatrix} \qquad \rm{and} \\
& \hspace{1cm} \mathcal{J}_n=
\begin{pmatrix}
	\rho_0 & \sqrt{d_1} & 0 & \ldots & 0 & 0\\
	\sqrt{d_1} & \rho_{1} &\sqrt{d_2} & \ldots & 0 & 0\\
	0 & \sqrt{d_2} & \rho_{2} & \ldots & 0 & 0\\
	\vdots & \vdots & \vdots & \ddots & \vdots & \vdots\\
	0 & 0 & 0 & \ldots & \rho_{n-2} & \sqrt{d_{n-1}}\\
	0 & 0 & 0 & \ldots & \sqrt{d_{n-1}} & \rho_{n-1}
\end{pmatrix}.
\end{align*}
\end{theorem}
\begin{proof}
The procedure is to divide the poles to form two new rational function sequences
\begin{align}
&u_0^{\mathcal{R}}(x)=\mathcal{L}_0(x), \quad u_k^{\mathcal{R}}(x)=\dfrac{(-1)^k}{(x-i\omega)^k\prod_{j=1}^{k}\sqrt{d_j}}\mathcal{L}_k(x), \label{u_k_R} \\
&u_0^{\mathcal{L}}(x)=\mathcal{L}_0(x), \quad u_k^{\mathcal{L}}(x)=\dfrac{(-1)^k}{(x+i\omega)^k\prod_{j=1}^{k}\sqrt{d_j}}\mathcal{L}_k(x), \label{u_k_L} \\
\text{and} \quad & \mathbf{u}^{\mathcal{R}}(x)=[u_0^{\mathcal{R}}(x) \ldots u_{n-1}^{\mathcal{R}}(x)]^T, \quad \mathbf{u}^{\mathcal{L}}(x)=[u_0^{\mathcal{L}}(x) \ldots u_{n-1}^{\mathcal{L}}(x)], \nonumber
\end{align}
which provides the components of the left and right eigenvectors of the infinite tridiagonal linear pencil $x\mathcal{J}-\mathcal{K}$. Thus
\begin{align*}
	x\mathcal{J}-\mathcal{K} = \begin{pmatrix}
		\zeta_{0}(x) &	-\phi_{1}^{R}(x) & 0 & 0 & \ldots  \\
		-\phi_{1}^{L}(x) &  \zeta_{1}(x) & -\phi_{2}^{R}(x) & 0 & \ldots  \\
		0 & -\phi_{2}^{L}(x) & \zeta_{2}(x) & -\phi_{3}^{R}(x) &  \ldots  \\
		\vdots	& \vdots & \vdots & \ddots & \vdots
		\end{pmatrix},
\end{align*}
where $\zeta_{j}(x)$, $\phi_{j}^{L}(x)$ and $\phi_{j}^{R}(x)$ are non-zero polynomials of degree one. Then, for $1 \leq k \leq n-1$, in view of \eqref{u_k_R} and \eqref{u_k_L}, \eqref{R2 recurrence from linear 1} can be written as
\begin{align}
&(x-i\omega)\sqrt{d_1}u_1^{\mathcal{R}}(x)+\rho_0(x-\alpha_1\rho_0^{-1}-c_0)u_0^{\mathcal{R}}(x)=0, \nonumber \\
&(x-i\omega)\sqrt{d_{k+1}}u_{k+1}^{\mathcal{R}}(x)+\rho_{k}(x+f_{k}^{-1}\rho_{k}^{-1}s_{k})u_{k}^{\mathcal{R}}(x)+(x+i\omega)\sqrt{d_{k}}u_{k-1}^{\mathcal{R}}(x)=0, \label{u_k_R Relation 1}
\end{align}
and
\begin{align}
&(x+i\omega)\sqrt{d_1}u_1^{\mathcal{L}}(x)+\rho_0(x-\alpha_1\rho_0^{-1}-c_0)u_0^{\mathcal{L}}(x)=0, \nonumber \\
&(x+i\omega)\sqrt{d_{k+1}}u_{k+1}^{\mathcal{L}}(x)+\rho_{k}(x+f_{k}^{-1}\rho_{k}^{-1}s_{k})u_{k}^{\mathcal{L}}(x)+(x-i\omega)\sqrt{d_{k}}u_{k-1}^{\mathcal{L}}(x)=0. \label{u_k_L Relation 1}
\end{align}
Expressions \eqref{u_k_R Relation 1} and \eqref{u_k_L Relation 1} gives
\begin{align}
&(\rho_0c_0+\alpha_1) u_0^{\mathcal{R}}+i\omega\sqrt{d_1}u_1^{\mathcal{R}}=x(\rho_0u_0^{\mathcal{R}}+\sqrt{d_1}u_1^{\mathcal{R}}), \nonumber\\
&-i\omega\sqrt{d_{k}}u_{k-1}^{\mathcal{R}}-f_{k}^{-1}s_{k}u_{k}^{\mathcal{R}}+i\omega\sqrt{d_{k+1}}u_{k+1}^{\mathcal{R}}=x(\sqrt{d_{k}}u_{k-1}^{\mathcal{R}}+\rho_ku_{k}^{\mathcal{R}}+\sqrt{d_{k+1}}u_{k+1}^{\mathcal{R}}), \label{u_k_R Relation 2}
\end{align}
and
\begin{align}
&(\rho_0c_0+\alpha_1) u_0^{\mathcal{L}}-i\omega\sqrt{d_1}u_1^{\mathcal{L}}=x(\rho_0u_0^{\mathcal{L}}+\sqrt{d_1}u_1^{\mathcal{L}}), \nonumber \\
&i\omega\sqrt{d_{k}}u_{k-1}^{\mathcal{L}}-f_{k}^{-1}s_{k}u_{k}^{\mathcal{L}}-i\omega\sqrt{d_{k+1}}u_{k+1}^{\mathcal{L}}=x(\sqrt{d_{k}}u_{k-1}^{\mathcal{L}}+\rho_ku_{k}^{\mathcal{L}}+\sqrt{d_{k+1}}u_{k+1}^{\mathcal{L}}). \label{u_k_L Relation 2}
\end{align}
Thus, with $1 \leq k \leq n-1$, \eqref{u_k_R Relation 2} implies \eqref{u_k_R Relation 3} and \eqref{u_k_L Relation 2} implies \eqref{u_k_L Relation 3} and the proof is complete.
\end{proof}

The following result is a generalization of \cite[Theorem 1.1]{Esmail Ranga 2018}. The proof is based on similar lines given in \Cref{Lemma 1 Biorthogonality section}.

\begin{theorem}\label{GEVP P_n}
Let $\mathbf{e}_n$ be the $n^{th}$ column of $n \times n$ identity matrix and $ \mathbf{u}(x)=[u_0(x), \ldots, u_{n-1}(x)]^T $ where
\begin{align*}
u_0(x)=\mathcal{P}_0(x), \quad u_k(x)=\dfrac{(-1)^k}{(x-i\omega)^k\prod_{j=1}^{k}\sqrt{d_j}}\mathcal{P}_k(x).
\end{align*}
Then, the matrix representation of \eqref{special R2 Linear pp} can be written as
\begin{align*}
\mathcal{G}_n\mathbf{u}(x)=x\mathcal{J}_n\mathbf{u}(x)+\sqrt{d_{n}}(x-i\omega)u_{n}(x)\mathbf{e}_n, 	
\end{align*}
where $\mathcal{G}_n$ is an $n\times n$ tridiagonal matrix of the form
\begin{align*}
\mathcal{G}_n=&
\begin{pmatrix}
	\rho_0c_0 & i\omega\sqrt{d_1} & 0 & \ldots & 0 & 0\\
	-i\omega\sqrt{d_1} & \rho_{1}c_1 & i\omega\sqrt{d_2} & \ldots & 0 & 0\\
	0 & -i\omega\sqrt{d_2} & \rho_{2}c_2 & \ldots & 0 & 0\\
	\vdots & \vdots & \vdots & \ddots & \vdots & \vdots\\
	0 & 0 & 0 & \ldots & \rho_{n-2}c_{n-2} & i\omega\sqrt{d_{n-1}}\\
	0 & 0 & 0 & \ldots & -i\omega\sqrt{d_{n-1}} & \rho_{n-1}c_{n-1}
\end{pmatrix},
%\qquad \rm{and} \\
%& \hspace{1cm}
%\mathcal{J}_n=
%\begin{pmatrix}
%	\rho_0 & \sqrt{d_1} & 0 & \ldots & 0 & 0\\
%	\sqrt{d_1} & \rho_{1} &\sqrt{d_2} & \ldots & 0 & 0\\
%	0 & \sqrt{d_2} & \rho_{2} & \ldots & 0 & 0\\
%	\vdots & \vdots & \vdots & \ddots & \vdots & \vdots\\
%	0 & 0 & 0 & \ldots & \rho_{n-2} & \sqrt{d_{n-1}}\\
%	0 & 0 & 0 & \ldots & \sqrt{d_{n-1}} & \rho_{n-1}
%\end{pmatrix},
\end{align*}
and $\mathcal{J}_n$ is given in \Cref{Lemma 1 Biorthogonality section}). Further, the zeros of $\mathcal{P}_n(x)$ are the eigenvalues of the generalized eigenvalue problem
\begin{align*}
	\mathcal{G}_n\mathbf{u}(x)=x\mathcal{J}_n\mathbf{u}(x).
\end{align*}
\end{theorem}

%The concept of biorthogonality is defined in several ways in the literature. In this manuscript, we consider the following idea given in \cite{Konhauser JAMA 1965}.

\subsection{Biorthogonality using Cholesky decomposition of $\mathcal{J}_n$:}
\begin{theorem}\label{Biorthogonality main theorem}
Let $x_j^{(n)}$, $j=1,2, \ldots n$ be the zeros of $\mathcal{L}_n(x)$. With the weight function $\omega_{n,j,k}$ given as
\begin{align}\label{w_n_j_k weight}
\omega_{n,j,k}^{-1}=-\sqrt{d_{n}}(x_k^{(n)}-i\omega)[u_{n}^{\mathcal{R}}(x_k^{(n)})]'u_{n-1}^{\mathcal{L}}(x_j^{(n)}) \neq 0, \quad j,k = 1 \ldots n.
\end{align}
and $\chi_i^{\mathcal{L}}(x)$ and $\chi_i^{\mathcal{R}}(x)$ defined by \eqref{chi_k_L} and \eqref{chi_k_R} respectively, the following biorthogonality relation holds
\begin{align}\label{Biorthogonality 1}
\sum_{i=0}^{n-1}\chi_i^{\mathcal{L}}(x_j^{(n)})\chi_i^{\mathcal{R}}(x_k^{(n)})\omega_{n,j,k}=\delta_{j,k}, \quad j,k = 1,2, \ldots n,
\end{align}
\end{theorem}

\begin{proof}
Pre-multiplying \eqref{u_k_L Relation 3} by $\mathbf{u}^{\mathcal{R}}(y)$ and post-multiplying \eqref{u_k_R Relation 3} by $\mathbf{u}^{\mathcal{L}}(x)$ after evaluating at $y$, we have
\begin{align}
&\mathcal{K}_n\mathbf{u}^{\mathcal{R}}(y)\mathbf{u}^{\mathcal{L}}(x)=y\mathcal{J}_n\mathbf{u}^{\mathcal{R}}(y)\mathbf{u}^{\mathcal{L}}(x)+\sqrt{d_{n}}(y-i\omega)u_{n}^{\mathcal{R}}(y)\mathbf{e}_n \mathbf{u}^{\mathcal{L}}(x), \label{u_k_R Relation 4} \\
&\mathbf{u}^{\mathcal{R}}(y)\mathbf{u}^{\mathcal{L}}(x)\mathcal{K}_n=x\mathbf{u}^{\mathcal{R}}(y)\mathbf{u}^{\mathcal{L}}(x)\mathcal{J}_n+\sqrt{d_{n}}(x+i\omega)\mathbf{u}^{\mathcal{R}}(y) u_{n}^{\mathcal{L}}(x)\mathbf{e}_n^{T}. \label{u_k_L Relation 4}
\end{align}
For any two matrices $\mathcal{M}$ and $\mathcal{N}$, $Tr(\mathcal{M}\mathcal{N})=Tr(\mathcal{N}\mathcal{M})$, provided both the products $\mathcal{M}\mathcal{N}$ and $\mathcal{N}\mathcal{M}$ exist. It is evident that $Tr(\mathcal{J}_n\mathbf{u}^{\mathcal{R}}(y)\mathbf{u}^{\mathcal{L}}(x))$ is equal to $Tr(\mathbf{u}^{\mathcal{R}}(y)\mathbf{u}^{\mathcal{L}}(x)\mathcal{J}_n)$, which is futher equal to the matrix product $\mathbf{u}^{\mathcal{L}}(x)\mathcal{J}_n \mathbf{u}^{\mathcal{R}}(y)$. Since the trace is a linear operator, subtracting the trace of \eqref{u_k_R Relation 4} from the trace of \eqref{u_k_L Relation 4}, we get
\begin{align}
\mathbf{u}^{\mathcal{L}}(x)\mathcal{J}_n \mathbf{u}^{\mathcal{R}}(y)=\dfrac{\sqrt{d_{n}}[(y-i\omega)u_{n}^{\mathcal{R}}(y)u_{n-1}^{\mathcal{L}}(x)-(x+i\omega)u_{n-1}^{\mathcal{R}}(y)u_{n}^{\mathcal{L}}(x)]}{x-y}. \label{u_L H_n u_R relation 1}
\end{align}
Using ${y \rightarrow x}$ in \eqref{u_L H_n u_R relation 1}, we obtain
\begin{align*}
\mathbf{u}^{\mathcal{L}}(x)\mathcal{J}_n \mathbf{u}^{\mathcal{R}}(x)=\sqrt{d_{n}}[(x+i\omega)\{u_{n-1}^{\mathcal{R}}(x)\}'u_{n}^{\mathcal{L}}(x)-(x-i\omega)\{u_{n}^{\mathcal{R}}(x)\}'u_{n-1}^{\mathcal{L}}(x)].
\end{align*}
Hence, at the zeros $x_j^{(n)}$, $j=1 \ldots n$,
\begin{align}
\mathbf{u}^{\mathcal{L}}(x_j^{(n)})\mathcal{J}_n \mathbf{u}^{\mathcal{R}}(x_j^{(n)})=-\sqrt{d_{n}}(x_j^{(n)}-i\omega)[u_{n}^{\mathcal{R}}(x_j^{(n)})]'u_{n-1}^{\mathcal{L}}(x_j^{(n)})=\omega_{n,j}^{-1}. \label{u_L H_n u_R relation 2}
\end{align}
Define $\omega_{n,j,k}^{-1}=-\sqrt{d_{n}}(x_k^{(n)}-i\omega)[u_{n}^{\mathcal{R}}(x_k^{(n)})]'u_{n-1}^{\mathcal{L}}(x_j^{(n)}) \neq 0$, $j,k = 1,2, \ldots n$ and note that $\omega_{n,j,j}^{-1}=\omega_{n,j}^{-1}$ for $k=j$. From \eqref{u_L H_n u_R relation 2}, for $j,k= 1 \ldots n$,
\begin{align*}
\left[\mathcal{J}_n^T[\mathbf{u}^{\mathcal{L}}(x_j^{(n)})]^T\right]^T \mathbf{u}^{\mathcal{R}}(x_k^{(n)})=[ \mathbf{u}^{\mathcal{R}}(x_k^{(n)})]^T \mathcal{J}_n^T [\mathbf{u}^{\mathcal{L}}(x_j^{(n)})]^T = \omega_{n,j,k}^{-1}\delta_{j,k}.
\end{align*}
Hence, it can be seen that two finite sequences $\{\mathcal{J}_n^T[\mathbf{u}^{\mathcal{L}}(x_j^{(n)})]^T \}_{j=1}^n$ and $\{\mathbf{u}^{\mathcal{R}}(x_k^{(n)}) \}_{k=1}^n$ are biorthogonal to each other.\par
Further, the Cholesky decomposition of the positive definite matrix $\mathcal{J}_n$ is such that $\mathcal{J}_n=\mathcal{C}_n\mathcal{C}_n^T$ where
\begin{align*}
\mathcal{C}_n=
\begin{pmatrix}
	m_0 & 0 & \ldots & 0 & 0\\
	\ell_1 & m_1 & \ldots & 0 & 0\\
	\vdots & \vdots & \ddots & \vdots & \vdots\\
	0 & 0 & \ldots & m_{n-2} & 0\\
	0 & 0 & \ldots & \ell_{n-1} & m_{n-1}
\end{pmatrix},
\end{align*}
and $m_i=\sqrt{\rho_i-\ell_i^2}$ (with $\ell_0=0$), $\ell_i=\sqrt{d_i}/m_{i-1}$, $i=0,1, \ldots n-1$. It follows from \eqref{u_L H_n u_R relation 2} that
\begin{align}\label{u_L cholesky u_R relation}
\mathbf{u}^{\mathcal{L}}(x_j^{(n)}) \mathcal{C}_n\mathcal{C}_n^T \mathbf{u}^{\mathcal{R}}(x_k^{(n)})=\left[\mathcal{C}_n^T[\mathbf{u}^{\mathcal{L}}(x_j^{(n)})]^T\right]^T [\mathcal{C}_n^T \mathbf{u}^{\mathcal{R}}(x_k^{(n)})]= \omega_{n,j,k}^{-1}\delta_{j,k}.
\end{align}
The two sequences $\{u_i^{\mathcal{L}}(x)\}_{i=0}^{n-1}$ and $\{u_i^{\mathcal{R}}(x)\}_{i=0}^{n-1}$ are now used to define two new sequences of rational functions \cite{Beckermann Derevyagin Zhedanov linear pencil 2010, Zhedanov biorthogonal JAT 1999} $\{\chi_i^{\mathcal{L}}(x)\}_{i=0}^{n-1}$ and $\{\chi_i^{\mathcal{R}}(x)\}_{i=0}^{n-1}$ given by
\begin{align}
&\chi_i^{\mathcal{L}}(x)= m_i u_{i}^{\mathcal{L}}(x)+\ell_{i+1}u_{i+1}^{\mathcal{L}}(x), \quad \chi_{n-1}^{\mathcal{L}}(x)= m_{n-1} u_{n-1}^{\mathcal{L}}(x), \label{chi_k_L} \\
& \chi_i^{\mathcal{R}}(x)= m_i u_{i}^{\mathcal{R}}(x)+\ell_{i+1}u_{i+1}^{\mathcal{R}}(x), \quad \chi_{n-1}^{\mathcal{R}}(x)= m_{n-1} u_{n-1}^{\mathcal{R}}(x). \label{chi_k_R}
\end{align}
Observe that $\chi_i^{\mathcal{L}}(x_j^{(n)})$ and $\chi_i^{\mathcal{R}}(x_k^{(n)})$ are the $i$-th component of the column vectors obtained from the matrix product $\mathcal{C}_n^T[\mathbf{u}^{\mathcal{L}}(x_j^{(n)})]^T$ and $\mathcal{C}_n^T[\mathbf{u}^{\mathcal{R}}(x_k^{(n)})]$, respectively. Thus, the system of equations resulting from \eqref{u_L cholesky u_R relation} can be written as
\begin{align*}
\sum_{i=0}^{n-1}\chi_i^{\mathcal{L}}(x_j^{(n)})\chi_i^{\mathcal{R}}(x_k^{(n)})\omega_{n,j,k}=\delta_{j,k}, \quad j,k = 1,2, \ldots n,
\end{align*}
which shows that the two finite sequences $\{\chi_i^{\mathcal{L}}(x)\}_{k=0}^{n-1}$ and $\{\chi_i^{\mathcal{R}}(x)\}_{k=0}^{n-1}$ are biorthogonal to each other on the set of zeros of $\mathcal{L}_n(x)$.
\end{proof}

\begin{remark}
The biorthogonality relation stated in \eqref{Biorthogonality 1} is between the two finite sequences $\{\chi_i^{\mathcal{L}}(x_j^{(n)})\}_{i=0}^{n-1}$ and $\{\chi_i^{\mathcal{R}}(x_k^{(n)})\}_{i=0}^{n-1}$ when the elements with the same index are evaluated at different zeros of $\mathcal{L}_n(x)$. This is different from the criterion of biorthogonality
\begin{align*}
\sum_{s=0}^{N}\Omega_s\psi_i(\lambda_s)\eta_k(\lambda_s)=\delta_{i,k}, \quad i,k=0,1,\ldots,N,
\end{align*}
described in \cite[eq. (3.6)]{Zhedanov biorthogonal JAT 1999} where the rational functions $\psi_i(\lambda)$ and $\eta_k(\lambda)$ with different indices are evaluated at the same zero $\lambda_s$, depicting the difference in the concept of biorthogonality between these two methods.
\end{remark}

\begin{remark}
Information about biorthogonal polynomials is abundant in the literature. Interested readers may refer to \cite{Hendriksen Njaastad Rocky J 1991, Konhauser JAMA 1965, Lubinsky Sidi JCAM 2022, Srivastava biorthogonal Laguerre 1982} for relevant discussions. Results on biothogonality relations between finite sequences of polynomials have been obtained in \cite{KKB Swami 2018, Silva Ranga LAA 2005}.
\end{remark}

\subsection{Biorthogonality using Darboux transformation and $\mathcal{L}\mathcal{D}\mathcal{U}$ decomposition of $\mathcal{J}_n$:} The Cholesky decomposition $\mathcal{J}_n=\mathcal{C}_n\mathcal{C}_n^T$ considered above is a $\mathcal{L}\mathcal{U}$ factorization, where $\mathcal{L}$ and $\mathcal{U}$ are the lower and upper triangular matrices, respectively. The technique of Darboux transformation (commutation methods) gives a $\mathcal{U}\mathcal{L}$ factorization of $\mathcal{J}_n$, which is used to derive a new biorthogonality relation. The transformation of a monic Jacobi matrix $\mathcal{J}$ defined as
\begin{align*}
\mathcal{J}=\mathcal{L}\mathcal{U} \rightarrow \mathcal{J}_\mathcal{C}= \mathcal{U}\mathcal{L},
\end{align*}
is called the Darboux transformation without parameter \cite{Bueno Marcell LAA 2004} or the Christoffel transformation of $\mathcal{J}$ \cite{Derevyagin Derkach 2011}. The $\mathcal{L}\mathcal{U}$ factorization of a monic Jacobi matrix $\mathcal{J}$ is unique, whereas being dependent on a free parameter, the $\mathcal{U}\mathcal{L}$ factorization is not unique. As a result, the decomposition
\begin{align}\label{UL of J_n}
\mathcal{J}_n=\mathcal{C}_n^T\mathcal{C}_n
\end{align}
is a $\mathcal{U}\mathcal{L}$ factorization of $\mathcal{J}_n$ with an additional assumption $\ell_n=0$ (see \cite[Lemma 3.1]{Bueno Marcell LAA 2004}), where $\ell_{i+1}=\sqrt{\rho_i-m_i^2}$ (with $m_0=s_0$ as a free parameter) and $m_{i+1}=\sqrt{d_i}/\sqrt{\rho_i-m_i^2}$, $i=0,1, \ldots n-1$.

%The two sequences $\{u_k^{\mathcal{L}}(x)\}_{k=0}^{n-1}$ and $\{u_k^{\mathcal{R}}(x)\}_{k=0}^{n-1}$ are now used to define two new sequences of rational functions $\{\mathcal{Y}_i^{\mathcal{L}}(x)\}_{i=0}^{n-1}$ and $\{\mathcal{Y}_i^{\mathcal{R}}(x)\}_{i=0}^{n-1}$ given by
%\begin{align}
%&\mathcal{Y}_i^{\mathcal{L}}(x)= \ell_{k+1}u_{k}^{\mathcal{L}}(x)+ m_{k+1} u_{k+1}^{\mathcal{L}}(x), \quad \mathcal{Y}_{0}^{\mathcal{L}}(x)= m_{0} u_{0}^{\mathcal{L}}(x), \label{Y_k_L} \\
%& \mathcal{Y}_i^{\mathcal{R}}(x)= \ell_{k+1}u_{k}^{\mathcal{R}}(x)+ m_{k+1} u_{k+1}^{\mathcal{R}}(x), \quad \mathcal{Y}_{0}^{\mathcal{R}}(x)= m_{0} u_{0}^{\mathcal{R}}(x). \label{Y_k_R}
%\end{align}

%Observe that $\mathcal{Y}_i^{\mathcal{L}}(x_j^{(n)})$ and $\mathcal{Y}_i^{\mathcal{R}}(x_k^{(n)})$ are the $i$-th component of the column vectors obtained from the matrix product $\mathcal{C}_n[\mathbf{u}^{\mathcal{L}}(x_j^{(n)})]^T$ and $\mathcal{C}_n \mathbf{u}^{\mathcal{R}}(x_k^{(n)})$, respectively. Thus, the system of equations resulting from \eqref{u_L Darboux u_R relation} can be written as
%\begin{align*}
%\sum_{i=0}^{n-1}\mathcal{Y}_i^{\mathcal{L}}(x_j^{(n)})\mathcal{Y}_i^{\mathcal{R}}(x_k^{(n)})\omega_{n,j,k}=\delta_{j,k}, \quad j,k = 1,2, \ldots n,
%\end{align*}
%which shows that the two finite sequences $\{\mathcal{Y}_i^{\mathcal{L}}(x)\}_{k=0}^{n-1}$ and $\{\mathcal{Y}_i^{\mathcal{R}}(x)\}_{k=0}^{n-1}$ are biorthogonal to each other on the set of zeros of $\mathcal{S}_n(x)$.

Another variant of Cholesky decomposition for $\mathcal{J}_n$, known as square-root free Cholesky decomposition \cite{Golub Van Loan 1996}, has the form $\mathcal{J}_n=\mathcal{S}_n\mathcal{D}_n\mathcal{S}_n^T$, where $\mathcal{D}_n$ is a diagonal matrix with diagonal elements $e_i$, $i=0,1,\ldots,n-1$ and $\mathcal{S}_n$ is a unit lower triangular matrix given by
\begin{align*}
\mathcal{S}_n=
\begin{pmatrix}
	1 & 0 & \ldots & 0 & 0\\
	\ell_1 & 1 & \ldots & 0 & 0\\
	\vdots & \vdots & \ddots & \vdots & \vdots\\
	0 & 0 & \ldots & 1 & 0\\
	0 & 0 & \ldots & \ell_{n-1} & 1
\end{pmatrix},
\end{align*}
with $e_0=\rho_0$, $\ell_1=\sqrt{d_1}/e_0$, $e_i=\rho_i-\dfrac{d_{i-1}}{e_{i-1}}$ and $\ell_{i+1}=\sqrt{d_{i+1}}/e_i$, $i=1,\ldots,n-1$ \cite{Meurant 1992}. With $\mathcal{S}^\mathcal{C}_n=\mathcal{S}_n\mathcal{D}^{1/2}_n$, we get
\begin{align}\label{LDU of J_n}
\mathcal{J}_n=\mathcal{S}_n\mathcal{D}_n\mathcal{S}_n^T=\mathcal{S}^\mathcal{C}_n[\mathcal{S}^\mathcal{C}_n]^T,
\end{align}
 where
\begin{align*}
\mathcal{S}^\mathcal{C}_n=
\begin{pmatrix}
	\sqrt{e_0} & 0 & \ldots & 0 & 0\\
	\sqrt{d_1/e_0} & \sqrt{e_1} & \ldots & 0 & 0\\
	\vdots & \vdots & \ddots & \vdots & \vdots\\
	0 & 0 & \ldots & \sqrt{e_{n-2}} & 0\\
	0 & 0 & \ldots & \sqrt{d_{n-1}/e_{n-2}} & \sqrt{e_{n-1}}
\end{pmatrix},
\end{align*}

\begin{theorem}
Let $x_j^{(n)}$, $j=1,2, \ldots n$ be the zeros of $\mathcal{L}_n(x)$
%\begin{align*}
%\omega_{n,j,k}^{-1}=-\sqrt{d_{n}}(x_k^{(n)}-i\omega)[u_{n}^{\mathcal{R}}(x_k^{(n)})]'u_{n-1}^{\mathcal{L}}(x_j^{(n)}) \neq 0, \quad j,k = 1 \ldots n.
%\end{align*}
and $\mathcal{Y}_i^{\mathcal{L}}(x)$, $\mathcal{Y}_i^{\mathcal{R}}(x)$, $\mathcal{Z}_i^{\mathcal{L}}(x)$ and $\mathcal{Z}_i^{\mathcal{R}}(x)$ be defined by \eqref{Y_k_L}, \eqref{Y_k_R}, \eqref{Z_k_L} and \eqref{Z_k_L}, respectively. Then, the following biorthogonality relations hold
\begin{align}
&\sum_{i=0}^{n-1}\mathcal{Y}_i^{\mathcal{L}}(x_j^{(n)})\mathcal{Y}_i^{\mathcal{R}}(x_k^{(n)})\omega_{n,j,k}=\delta_{j,k}, \quad j,k = 1,2, \ldots n \quad \mbox{and} \label{Biorthogonality 2} \\
&\sum_{i=0}^{n-1}\mathcal{Z}_i^{\mathcal{L}}(x_j^{(n)})\mathcal{Z}_i^{\mathcal{R}}(x_k^{(n)})\omega_{n,j,k}=\delta_{j,k}, \quad j,k = 1,2, \ldots n, \label{Biorthogonality 3}
\end{align}
where the weight function $\omega_{n,j,k}$ is given by \eqref{w_n_j_k weight}.
\end{theorem}
\begin{proof}
Proceeding as per the proof of \Cref{Biorthogonality main theorem} and using the decompositions of $\mathcal{J}_n$ given by \eqref{UL of J_n} and \eqref{LDU of J_n}, it follows from \eqref{u_L H_n u_R relation 2} that
\begin{align}
&\mathbf{u}^{\mathcal{L}}(x_j^{(n)}) \mathcal{C}_n^T\mathcal{C}_n \mathbf{u}^{\mathcal{R}}(x_k^{(n)})=\left[\mathcal{C}_n[\mathbf{u}^{\mathcal{L}}(x_j^{(n)})]^T\right]^T [\mathcal{C}_n \mathbf{u}^{\mathcal{R}}(x_k^{(n)})]= \omega_{n,j,k}^{-1}\delta_{j,k}. \label{u_L Darboux u_R relation} \\
&\mathbf{u}^{\mathcal{L}}(x_j^{(n)}) \mathcal{S}^\mathcal{C}_n[\mathcal{S}^\mathcal{C}_n]^T \mathbf{u}^{\mathcal{R}}(x_k^{(n)})=\left[[\mathcal{S}^\mathcal{C}_n]^T[\mathbf{u}^{\mathcal{L}}(x_j^{(n)})]^T\right]^T \left[[\mathcal{S}^\mathcal{C}_n]^T \mathbf{u}^{\mathcal{R}}(x_k^{(n)})\right]= \omega_{n,j,k}^{-1}\delta_{j,k}. \label{u_L LDL u_R relation}
\end{align}
Defining new sequences of rational functions $\{\mathcal{Y}_i^{\mathcal{L}}(x)\}_{i=0}^{n-1}$, $\{\mathcal{Y}_i^{\mathcal{R}}(x)\}_{i=0}^{n-1}$, $\{\mathcal{Z}_i^{\mathcal{L}}(x)\}_{i=0}^{n-1}$ and $\{\mathcal{Z}_i^{\mathcal{R}}(x)\}_{i=0}^{n-1}$ with the help of $\{u_i^{\mathcal{L}}(x)\}_{i=0}^{n-1}$ and $\{u_i^{\mathcal{R}}(x)\}_{i=0}^{n-1}$, we obtain
\begin{align}
&\mathcal{Y}_i^{\mathcal{L}}(x)= \ell_{i+1}u_{i}^{\mathcal{L}}(x)+ m_{i+1} u_{i+1}^{\mathcal{L}}(x), \quad \mathcal{Y}_{0}^{\mathcal{L}}(x)= m_{0} u_{0}^{\mathcal{L}}(x), \label{Y_k_L} \\
& \mathcal{Y}_i^{\mathcal{R}}(x)= \ell_{i+1}u_{i}^{\mathcal{R}}(x)+ m_{i+1} u_{i+1}^{\mathcal{R}}(x), \quad \mathcal{Y}_{0}^{\mathcal{R}}(x)= m_{0} u_{0}^{\mathcal{R}}(x). \label{Y_k_R} \\
&\mathcal{Z}_i^{\mathcal{L}}(x)= \sqrt{e_i} u_{i}^{\mathcal{L}}(x)+\sqrt{d_{i+1}/e_{i}}u_{i+1}^{\mathcal{L}}(x), \quad \mathcal{Z}_{0}^{\mathcal{L}}(x)= \sqrt{e_{n-1}} u_{n-1}^{\mathcal{L}}(x), \label{Z_k_L} \\
& \mathcal{Z}_i^{\mathcal{R}}(x)= \sqrt{e_i} u_{i}^{\mathcal{R}}(x)+\sqrt{d_{i+1}/e_{i}}u_{i+1}^{\mathcal{R}}(x), \quad \mathcal{Z}_{0}^{\mathcal{R}}(x)= \sqrt{e_{n-1}} u_{n-1}^{\mathcal{R}}(x). \label{Z_k_R}
\end{align}
Note that $\mathcal{Y}_i^{\mathcal{L}}(x_j^{(n)})$, $\mathcal{Y}_i^{\mathcal{R}}(x_k^{(n)})$, $\mathcal{Z}_i^{\mathcal{L}}(x_j^{(n)})$ and $\mathcal{Z}_i^{\mathcal{R}}(x_k^{(n)})$ are the $i$-th component of the column vectors obtained from the matrix product $\mathcal{C}_n[\mathbf{u}^{\mathcal{L}}(x_j^{(n)})]^T$, $\mathcal{C}_n \mathbf{u}^{\mathcal{R}}(x_k^{(n)})$, $[\mathcal{S}^\mathcal{C}_n]^T[\mathbf{u}^{\mathcal{L}}(x_j^{(n)})]^T$ and $[\mathcal{S}^\mathcal{C}_n]^T \mathbf{u}^{\mathcal{R}}(x_k^{(n)})$ respectively. Thus, the system of equations resulting from \eqref{u_L Darboux u_R relation} and \eqref{u_L LDL u_R relation} is given by \eqref{Biorthogonality 2} and \eqref{Biorthogonality 3}. This completes the proof.
\end{proof}

\begin{remark}
The method for proving biorthogonality between rational functions used here differs from that used in \cite[Theorem 1]{Zhedanov biorthogonal JAT 1999}. The Christoffel-Darboux type formula \eqref{u_L H_n u_R relation 1} is a fundamental component of our approach.
\end{remark}

\section{Interlacing Properties}\label{Interlacing Properties}
The following results on the zeros of $R_{II}$ polynomials satisfying \eqref{special R2 Linear pp} will be helpful for further discussion.

\begin{theorem}\cite[Theorem 6.1]{swami vinay R2 2022}\label{Ranga leading coefficient}
The polynomial $\mathcal{P}_n(x)$ obtained from \eqref{GCRR recurrence} is of degree $n$ with positive leading coefficient. Further, if we denote by $\mathfrak{k}_n$ the leading coefficient of $\mathcal{P}_n(x)$, then $\mathfrak{k}_0 = 1$, $\mathfrak{k}_1 = 1$ and
\begin{align*}
	0 < (1-\ell_{n-1})=\dfrac{\mathfrak{k}_n}{\mathfrak{k}_{n-1}} < 1, \quad n \geq 2,
\end{align*}
where $\{\ell_{n}\}_{n \geq 0}$ is the minimal parameter sequence of the positive chain sequence $\{d_{n}\}_{n \geq 1}$.
\end{theorem}

\begin{theorem}\cite[Theorem 6.2]{swami vinay R2 2022}\label{Ranga zeros interlacing R2}
The zeros $x^{(n)}_j$, $j=1,2, \ldots n$ of $\mathcal{P}_n(x)$, $n \geq 0$ derived from \eqref{GCRR recurrence}, are real and simple. Assuming the ordering $x^{(n)}_{j-1} < x^{(n)}_j$ for the zeros, we have the interlacing
\begin{align*}
x^{(n+1)}_1 < x^{(n)}_{1} < x^{(n+1)}_2 < \ldots < x^{(n+1)}_n < x^{(n)}_{n} < x^{(n+1)}_{n+1}, \quad  n \geq 1.
%x^{(n+1)}_{n+1} < x^{(n)}_{n}<x^{{n+1}}_n < \ldots < x^{(n+1)}_2 < x^{(n)}_{1} <x^{(n+1)}_1, \quad  n \geq 1.
\end{align*}
\end{theorem}

\begin{remark}
For $w=1$, \Cref{Ranga leading coefficient} and \Cref{Ranga zeros interlacing R2} have been established in \cite{Esmail Ranga 2018}.
\end{remark}

A direct consequence of \Cref{Lemma 1 Biorthogonality section}, \Cref{GEVP P_n} and \Cref{Ranga zeros interlacing R2} is the following:
\begin{theorem}
The zeros of polynomials $\mathcal{L}_n(x)$, $n \geq 1$, are real and simple. Moreover, they are the eigenvalues of the generalised eigenvalue problem
\begin{align*}
	\mathcal{K}_n\mathbf{u}^{\mathcal{R}}(x)=x\mathcal{J}_n\mathbf{u}^{\mathcal{R}}(x),
\end{align*}
where $\mathcal{K}_n$ and $\mathcal{J}_n$ are given in \Cref{Lemma 1 Biorthogonality section}.
\end{theorem}

%Therefore, the zeros of $\mathcal{S}_n(x)$, $n \geq 0$, obtained from \eqref{R2 recurrence from linear 1} are also real and simple.
It is interesting to investigate the location of zeros of $\mathcal{L}_n(x)$ with respect to those of $\mathcal{P}_n(x)$ and $\mathcal{P}_{n-1}(x)$. Such explorations are already available in the literature in the context of quasi-orthogonality of classical orthogonal polynomials and are often addressed as triple interlacing of zeros \cite{Brezinski Driver Redivo-Zaglia ANM 2004, Joulak ANM 2005}. As a result, we arrive at the following conclusion.
%INTERLACING RN RN1 RESULT
\begin{theorem}\label{Interlacing theorem P_n P_n-1 Q_n}
The zeros of $\mathcal{P}_n(x)$, $\mathcal{P}_{n-1}(x)$ and $\mathcal{L}_n(x)$ denoted as $x_{i}^{(n)}$, $x_{i}^{(n-1)}$ and $y_i$ respectively, satisfy triple interlacing property as follows: \\
{\rm (1)} For $\alpha_{n}>0$, $x_{i}^{(n)} < y_i < x_{i}^{(n-1)}$, $\forall i=1,2, \ldots, n-1$ on $\mathbb{R} \backslash [x_{n}^{(n)}, \infty]$.\\
{\rm (2)} For $\alpha_{n}<0$, $x_{i-1}^{(n-1)} < y_i < x_{i}^{(n)}$, $\forall i=2, \ldots, n$ on $\mathbb{R} \backslash [-\infty,x_{1}^{(n)}]$.
\end{theorem}

\begin{proof}
We prove (1) given above and (2) follows by similar argument. Assume that $n$ is an odd number, as the situation for even $n$ can be handled similarly. Using \Cref{Ranga leading coefficient} and the interlacting property of zeros of $\mathcal{P}_n(x)$ and $\mathcal{P}_{n-1}(x)$, we conclude that
\begin{align}
&	(-1)^{i}\mathcal{P}_n(x_{i}^{(n-1)})< 0, \quad {\rm and} \quad (-1)^{i}\mathcal{P}_{n-1}(x_{i}^{(n)})< 0. \label{Interlacing P_n-1 relation}
%&	\mathcal{P}_{n-1}(x_{2j-1}^{(n)})< 0, \quad \mathcal{P}_{n-1}(x_{2j}^{(n)})> 0,
\end{align}
%&	\mathcal{P}_n(x_{2j-1}^{(n-1)})> 0, \quad \mathcal{P}_n(x_{2j}^{(n-1)})< 0, \quad {\rm and} \quad \mathcal{P}_{n-1}(x_{2j-1}^{(n)})> 0, \quad \mathcal{P}_{n-1}(x_{2j}^{(n)})< 0,
%\end{align*}
%or, equivalently,
%\begin{align}
%&	\mathcal{P}_n(x_{1}^{(n-1)})> 0, \quad \mathcal{P}_n(x_{2}^{(n-1)})< 0, ~ \ldots, ~ \mathcal{P}_n(x_{n-2}^{(n-1)})< 0, \quad \mathcal{P}_n(x_{n-1}^{(n-1)})> 0, \label{Interlacing P_n relation 1} \\
%&	\mathcal{P}_{n-1}(x_{1}^{(n)})> 0, \quad \mathcal{P}_{n-1}(x_{2}^{(n)})< 0, ~ \ldots, ~ \mathcal{P}_{n-1}(x_{n-2}^{(n)})< 0, \quad \mathcal{P}_{n-1}(x_{n-1}^{(n)})> 0. \label{Interlacing P_n-1 relation 1}
%\end{align}
Then, using \eqref{Interlacing P_n-1 relation} in \eqref{Linear combination Q_n}, for $\alpha_{n} > 0$, we obtain
\begin{align*}
&(-1)^{i}\mathcal{L}_n(x_{i}^{(n)})>0, \quad (-1)^{i+1}\mathcal{L}_n(x_{i}^{(n-1)})>0,
%&\mathcal{S}_n(x_{1}^{(n)})<0, \quad \mathcal{S}_n(x_{1}^{(n-1)})>0, \quad \mathcal{S}_n(x_{2}^{(n)})>0, \quad \mathcal{S}_n(x_{2}^{(n-1)})<0, \quad \ldots \\
%& \mathcal{S}_n(x_{n-2}^{(n)})>0, \quad \mathcal{S}_n(x_{n-2}^{(n-1)})<0, \quad \mathcal{S}_n(x_{n-1}^{(n)})<0, \quad \mathcal{S}_n(x_{n-1}^{(n-1)})>0, \quad \mathcal{S}_n(x_{n}^{(n)})<0,
\end{align*}
which implies
\begin{align*}
\mathcal{L}_n(x_{i}^{(n)})\mathcal{L}_n(x_{i}^{(n-1)})<0, \quad i=1,2,\ldots,n-1.
% \mathcal{S}_n(x_{2}^{(n)})\mathcal{S}_n(x_{2}^{(n-1)})<0, ~, \ldots, ~ \mathcal{S}_n(x_{n-1}^{(n)})\mathcal{S}_n(x_{n-1}^{(n-1)})<0.
\end{align*}
By intermediate value property, we have
\begin{align*}
x_{1}^{(n)} < y_1 < x_{1}^{(n-1)} < x_{2}^{(n)} < y_2 < x_{2}^{(n-1)}< \ldots < x_{n-1}^{(n)} < y_{n-1} < x_{n-1}^{(n-1)}. \hspace{3.3cm} \qedhere
\end{align*}
\end{proof}
%The other eventuality also follows from similar arguments.

It is easy to see that different self perturbed sequences of $R_{II}$ polynomials can be designed by changing the sequence  $\{\alpha_n\}_{n=1}^\infty$. Any such sequence under the assumptions of \Cref{Theorem 1 with aplha_n condition} should satisfy the recurrence relation \eqref{R2 recurrence from linear 1} and thus has real and simple zeros by virtue of \Cref{Ranga zeros interlacing R2}. It is worthy to observe the interlacing of zeros between two such sequences of polynomials. Results on the interlacing of zeros of some orthogonal polynomials from different sequences can be found in \cite{Driver Jordaan Mbuyi ANM 2009, Jordaan Tookos 2009}. Wronskians are useful objects in discussions related to the linear combination of polynomials. Recall that Wronskian of the two polynomials $\mathcal{U}_n(x)$ and $\mathcal{V}_n(x)$ given by
\begin{align*}
\mathcal{W}_n(x)=\begin{vmatrix}
	\mathcal{U}_n(x) & \mathcal{V}_{n}(x)\\
	\mathcal{U'}_n(x) & \mathcal{V'}_{n}(x)
\end{vmatrix},
\end{align*}
has the following properties:
\begin{enumerate}
\item $\mathcal{W}_n(x)$ is linear in $\mathcal{U}_n(x)$ and $\mathcal{V}_n(x)$.
\item If $\mathcal{W}_n(x) \neq 0$ in some interval $\mathcal{I}$, then $\mathcal{U}_n(x)$ and $\mathcal{V}_n(x)$ have simple and intelacing zeros in $\mathcal{I}$.
\end{enumerate}

\begin{theorem}\label{Theorem different seq interlace}
Let $\mathcal{L}_n(x)$, $n \geq 0$, be as defined by \eqref{Linear combination Q_n}. Further, let $\mathcal{T}_n(x)=\mathcal{P}_n(x)-\beta_n\mathcal{P}_{n-1}(x)$, $\beta_n \in \mathbb{R}\backslash \{ 0 \}$, $n \geq 0$, be another class of $R_{II}$ polynomials (as defined in \Cref{remark another alpha_n or beta_n}). Then, for $\alpha_{n} \neq \beta_n$, $\mathcal{L}_n(x)$ and $\mathcal{T}_n(x)$ have no common zeros and their zeros interlace.
\end{theorem}

\begin{proof}
It is known that for $R_{II}$ polynomials $\mathcal{P}_{n}(x)$ and $\mathcal{P}_{n-1}(x)$, the Wronskian $\mathcal{W}_n(x)$
\begin{align*}
\mathcal{W}_n(x)=\begin{vmatrix}
\mathcal{P}_{n-1}(x) & \mathcal{P}_{n}(x)\\
\mathcal{P'}_{n-1}(x) & \mathcal{P'}_{n}(x)
\end{vmatrix},
\end{align*}
is positive for all real $x$. By the linearity property of Wronskian, we obtain
\begin{align*}
\mathcal{W}^1_n(x)=\begin{vmatrix}
	\mathcal{L}_n(x) & \mathcal{T}_{n}(x)\\
	\mathcal{L'}_n(x) & \mathcal{T'}_{n}(x)
\end{vmatrix}=(\beta_n-\alpha_{n})\mathcal{W}_n(x) \neq 0,
\end{align*}
whenever $\alpha_{n} \neq \beta_n$. Thus, the theorem follows from the properties of Wronskian.
\end{proof}

\begin{remark}
Let $\mathcal{L}_n(x)=\gamma_{n}\mathcal{P}_n(x)-\alpha_n\mathcal{P}_{n-1}(x)$ and $\mathcal{T}_n(x)=\epsilon_n\mathcal{P}_n(x)-\beta_n\mathcal{P}_{n-1}(x)$, $\alpha_n, \beta_n, \gamma_{n}, \epsilon_n \in \mathbb{R}\backslash \{ 0 \}$, $n \geq 0$. It is easy to verify that the conclusion of \Cref{Theorem different seq interlace} holds whenever $\alpha_n \epsilon_n -\beta_n \gamma_{n} \neq 0 $.
\end{remark}

Results concerning the self perturbation (or linear combination) of polynomials involve a constant sequence that combines two (or more) polynomials. In this section, we establish the existence of a non-constant unique sequence that satisfies both conditions \eqref{condition 1} and \eqref{condition 2}. The scaled version of GCRR polynomials \eqref{GCRR polynomial} is used for this purpose, which is given as
\begin{align}
\mathcal{P}'_{n}(x)=\dfrac{(\zeta)_n}{(2\zeta)_n}\mathcal{P}_{n}(x), \quad n \geq 1,
\end{align}
and satisfying the $R_{II}$ type recurrence relation
\begin{align}\label{CRR recurrence}
\mathcal{P}'_{n+1}(x)=\dfrac{\zeta+n}{2\zeta+n} \left(x-\dfrac{\theta}{\zeta+n}\right)\mathcal{P}'_n(x)-\dfrac{n}{4(2\zeta+n)} (x^2+\omega^2)\mathcal{P}'_{n-1}(x), \quad  n \geq 1.
\end{align}
Clearly,
\begin{align*}
\rho_{n}=\dfrac{\zeta+n}{2\zeta+n}, \quad c_{n-1} = \dfrac{\theta}{\zeta+n-1}, \qquad d_{n} = \dfrac{n}{4(2\zeta+n)}, \quad n \geq 1.
\end{align*}

\begin{remark}
The conditions \eqref{condition 1} and \eqref{condition 2} suggest that $\alpha_n$, $n \geq 1$, should be a root of the quadratic equation
\begin{align}\label{Quadratic equation 1}
\rho_n x^2-\rho_{n}\rho_{n-1}(1-c_{n-1})x+\omega^2d_n\rho_{n-1}=0, \quad n\geq 1,
\end{align}
for $\mathcal{L}_n(x)$, $n \geq 1$, to be a sequence of $R_{II}$ polynomials.
\end{remark}

The sequence $\{\alpha_n\}_{n=1}^\infty$, for self perturbing GCRR polynomials, is obtained as a solution to the quadratic equation \eqref{Quadratic equation 1}, i.e.,
\begin{align}\label{alpha_n as a root}
\alpha_n = \dfrac{\rho_{n}\rho_{n-1}(1-c_{n-1})\pm \sqrt{\rho_{n}^2\rho_{n-1}^2(1-c_{n-1})^2-4\omega^2d_n\rho_{n}\rho_{n-1}}}{2\rho_{n}}, \quad n \geq 1.
\end{align}
Observe that, for every $n$, there are two choices of $\alpha_n$. This gives rise to infinite possible ways to select $\{\alpha_n\}_{n=1}^\infty$. To have a unique choice for $\alpha_n$, one of the possible ways is to assume the discriminant in \eqref{alpha_n as a root} to be zero, i.e.,
\begin{align}\label{B^2-4AC}
&\rho_{n}^2\rho_{n-1}^2(1-c_{n-1})^2-4\omega^2d_n\rho_{n}\rho_{n-1}= 0 \nonumber\\
\Longrightarrow & \left(\dfrac{\zeta+n}{2\zeta+n}\right) \left(\dfrac{\zeta+n-1}{2\zeta+n-1}\right) \left(\dfrac{\zeta+n-\theta-1}{\zeta+n-1}\right)^2-\dfrac{n\omega^2}{2\zeta+n}=0.
\end{align}
Without loss of generality, we can assume $\theta=0$. So, \eqref{B^2-4AC} becomes
\begin{align}
\zeta^2+[2n(1-\omega^2)-1]\zeta+(1-\omega^2)(n^2-n)=0
\end{align}
The roots of this equation have to be positive ($\because \zeta > 0$), which restricts the choice of $\omega$, i.e., $\omega \geq 1$.
Now, $\{\alpha_n\}_{n=0}^\infty$ can be uniquely determined from \eqref{alpha_n as a root}.\par \noindent \textbf{Note:} For $\omega =1$, we have
\begin{align}\label{alpha_n}
\alpha_n = \dfrac{\rho_{n-1}(1-c_{n-1})}{2}=\dfrac{n}{2(n+1)}, \quad n \geq 1.
\end{align}

Based on the the \Cref{remark another alpha_n or beta_n}, the following can be proposed:
\begin{remark}
If the elements of the sequence $\beta_n$, $n \geq 1$, are the solutions of the quadratic equation
\begin{align}\label{Quadratic equation 2}
\rho_n x^2+\rho_{n}\rho_{n-1}(1+c_{n-1})x+\omega^2d_n\rho_{n-1}=0, \quad n\geq 1,
\end{align}
then, $\mathcal{T}_n(x)$, $n \geq 1$, is a sequence of $R_{II}$ polynomials.
\end{remark}
Following the technique used for the construction of $\{\alpha_n\}_{n=0}^\infty$, for the same choice of $\zeta$ and $\omega$, we conclude that
\begin{align}\label{beta_n}
\beta_n = -\dfrac{\rho_{n-1}(1+c_{n+1})}{2}=-\dfrac{n}{2(n+1)}, \quad n \geq 1.
\end{align}
%. The sequences $\{\alpha_n\}_{n=0}^\infty$ and $\{\beta_n\}_{n=0}^\infty$ constructed in this section are used to study the interlacing property between two self perturbed GCRR polynomials.

Now, we will see the illustration of concepts developed in this section. For this purpose, the GCRR polynomials defined in \Cref{Recurrence Relations of R2 type} are considered. Depending upon the sign of the sequence, combining the adjacent polynomials, two scenarios of interlacing arise. These two cases are exhibited when the sequences $\{\alpha_n\}_{n=1}^\infty$ ($\alpha_{n} > 0$, $n \geq 1$) and $\{\beta_n\}_{n=1}^\infty$ ($\beta_{n} < 0$, $n \geq 1$) determined previously (see \eqref{alpha_n} and \eqref{beta_n}) are used to generate two different sequences of $R_{II}$ polynomials $\mathcal{L}_n(x)$ and $\mathcal{T}_n(x)$, respectively. The location of the zeros of $\mathcal{P}_n(x)$, $\mathcal{P}_{n-1}(x)$ and $\mathcal{L}_n(x)$ for $n=8$ is depicted in \Cref{Fig1}.
\begin{figure}[H]
\includegraphics[scale=0.8]{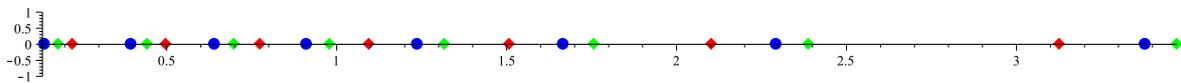}
\caption{Zeros of $\mathcal{P}_{7}(x)$ (red diamonds), $\mathcal{P}_{8}(x)$ (blue circles) and $\mathcal{L}_{8}(x)$ (green squares) }
\label{Fig1}
\end{figure}
The position of the zeros of $\mathcal{T}_n(x)$ with respect to $\mathcal{P}_n(x)$ and $\mathcal{P}_{n-1}(x)$  for $n=8$ is portrayed in \Cref{Fig2}.
\begin{figure}[H]
\includegraphics[scale=0.8]{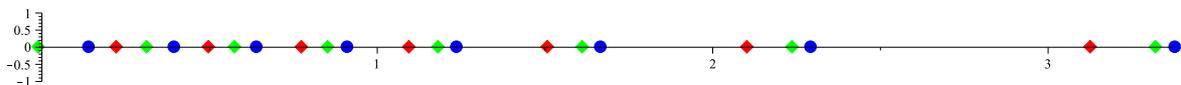}
\caption{Zeros of $\mathcal{P}_{7}(x)$ (red diamonds), $\mathcal{P}_{8}(x)$ (blue circles) and $\mathcal{T}_{8}(x)$ (green squares)}
\label{Fig2}
\end{figure}
Clearly, for both cases, the observations made here are consistent with the assertions of \Cref{Interlacing theorem P_n P_n-1 Q_n}. All the computations are performed and the plots for the zeros are provided using Mathematica ${\textregistered}$ with Intel Core i3-6006U CPU @ 2.00 GHz. \par

On the other hand, we inspect the situation of zeros from two different sequences of $R_{II}$ polynomials, namely $\mathcal{L}_n(x)$ and $\mathcal{T}_n(x)$. The interlacing between zeros, as shown in \Cref{Fig3}, confirms the validity of \Cref{Theorem different seq interlace}.
\begin{figure}[H]
\includegraphics[scale=0.8]{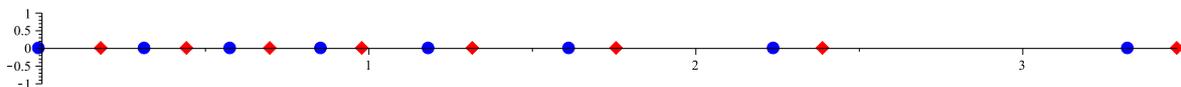}
\caption{Zeros of $\mathcal{L}_{8}(x)$ (red diamonds) and $\mathcal{T}_{8}(x)$ (blue circles)}
\label{Fig3}
\end{figure}

%\section{An illustration}\label{An illustration}
It is difficult to find closed-form expressions for $\mathcal{L}_n(x)=\mathcal{P}_n(x)-\alpha_n\mathcal{P}_{n-1}(x)$ and $\mathcal{T}_n(x)=\mathcal{P}_n(x)-\beta_n\mathcal{P}_{n-1}(x)$ where $\{\mathcal{P}_n(x)\}_{n=0}^\infty$ are GCRR polynomials. However, we will see that a closed expression can be obtained under some specific assumptions. Consider $R_{II}$ type recurrence \eqref{special R2 Linear pp} with $\rho_n = 1$ $c_n = 0$, $n \geq 0$ and $d_n =1/4$, $n \geq 1$. Note that such a choice of parameters corresponds to the GCRR polynomials when $\zeta=1$. The sequence $\{d_n\}_{n \geq 1}$ is a positive chain sequence.  Solving the recurrence
\begin{align*}
&\mathcal{P}_{n+1}(x) = x\mathcal{P}_n(x)-\dfrac{1}{4} (x^2+\omega^2)\mathcal{P}_{n-1}(x), \quad n\geq 1, \\
& \mathcal{P}_{1}(x)=x, \quad \mathcal{P}_{0}(x)=1,
\end{align*}
with above assumptions, we get
\begin{align*}
&\mathcal{P}_{n}(x)= \dfrac{i}{\omega}\left(\dfrac{x-i\omega}{2}\right)^{n+1}-\dfrac{i}{\omega}\left(\dfrac{x+i\omega}{2}\right)^{n+1}, \quad n \geq 0.
\end{align*}
Consider the self perturbation of polynomials $\{\mathcal{P}_n(x)\}_{n=0}^\infty$ defined as $\mathcal{L}_n(x)=\mathcal{P}_n(x)-\alpha_n\mathcal{P}_{n-1}(x)$, $ n\geq 1$. Our aim is to find a sequence $\{\alpha_n\}_{n=0}^\infty$ satisfying \eqref{condition 1} and \eqref{condition 2} so that $\{\mathcal{L}_n(x)\}_{n=0}^\infty$ is a sequence of $R_{II}$ polynomials. With the choice of $\rho_n$ and $d_n$ made above, from \eqref{condition 1}, we have $4\alpha_n\alpha_{n-1}-4\alpha_{n-1}+\omega^2=0$. Now, \eqref{condition 2} implies $\alpha_{n-1}=\alpha_n$, $n \geq 2$, which means $\{\alpha_n\}_{n=0}^\infty$ is a constant sequence. Since $\alpha_n$ reduces to a constant, we find that $\alpha_n=1/2\pm\dfrac{\sqrt{1-\omega^2}}{2}=\kappa$ (say). Thus, the linear combination
%Assume $\alpha_n=\kappa$ which implies $\kappa=1/2\pm\dfrac{\sqrt{1-\omega^2}}{2}$. Therefore, we must choose $\omega = \pm 1$.
\begin{align}\label{Illustration Q_n linear combi}
\mathcal{L}_n(x)=\dfrac{i}{\omega}\left[\left(\dfrac{x-iw}{2}\right)^{n+1}-\left(\dfrac{x+iw}{2}\right)^{n+1}\right]- \dfrac{i \kappa}{\omega}\left[\left(\dfrac{x-iw}{2}\right)^{n}- \left(\dfrac{x+iw}{2}\right)^{n}\right],
\end{align}
satisfies $R_{II}$ type recurrence relation \eqref{R2 recurrence from linear 1}. With $\alpha_n=1/2$, $ n \geq 0$, the constants in \Cref{Theorem Linear R2 general RR} are given as
\begin{align*}
&e_n=p_n= 1/4, \quad f_n=q_n=-\kappa, \quad g_n=\kappa^2+\omega^2/4, \quad r_n=\kappa^2, \quad t_n=-1/16,\\
& s_n=0, \quad u_n=\kappa/4, \quad w_n=u_n\omega^2, \quad v_n =-(\kappa^2/4+\omega^2/8), \quad z_n=-\omega^2g_n/4, \quad n \geq 1.
\end{align*}
So, from \eqref{R2 recurrence from linear 1}, we have
\begin{align*}
&	\mathcal{L}_{n+1}(x) = x\mathcal{L}_n(x)-\dfrac{1}{4} (x^2+\omega^2)\mathcal{L}_{n-1}(x), \quad n\geq 1, \\
& \mathcal{L}_{1}(x)=x-\kappa, \quad \mathcal{L}_{0}(x)=1.
\end{align*}
This recurrence relation can be solved to obtain
\begin{align*}
\mathcal{L}_{0}(x)=1, \quad \mathcal{L}_n(x)=\dfrac{i}{\omega2^{n+1}}\left[(x-i\omega)^n(x-i\omega-2\kappa)-(x+i\omega)^n(x+i\omega-2\kappa) \right], \quad n\geq 1,
\end{align*}
which is same as \eqref{Illustration Q_n linear combi}. \par
%To discuss the interlacing properties of zeros, we start with constructing a new sequence $\mathcal{R}_n(x)=\mathcal{P}_n(x)-\beta_n\mathcal{P}_{n-1}(x)$ of $R_{II}$ polynomials.
Note that, for $\mathcal{T}_n(x)$, we can choose $\omega=\pm 1$ and the sequence $\{\beta_n\}_{n=1}^\infty$ as per the conditions mentioned in \Cref{remark another alpha_n or beta_n}. Since, $\rho_n$ and $d_n$ are same as above, we have $\beta_{n-1}=\beta_n=\beta$, $n \geq 2$, from \eqref{condition 2}. Further, the restriction $f_n=g_n$ implies $\beta$ is a root of $4\beta^2+4\beta+1=0$ $\Rightarrow \beta=-1/2$. Thus, the polynomials $\mathcal{T}_n(x)$ are given as
\begin{align*}
\mathcal{T}_{0}(x)=1, \quad \mathcal{T}_n(x)=\dfrac{i}{2^{n+1}}\left[(x-i)^n(x-i+1)-(x+i)^n(x+i+1) \right], \quad n\geq 1,
\end{align*}
and they satisfy the $R_{II}$ type recurrence
\begin{align*}
&	\mathcal{T}_{n+1}(x) = x\mathcal{T}_n(x)-\dfrac{1}{4} (x^2+1)\mathcal{T}_{n-1}(x), \quad n\geq 1, \\
& \mathcal{T}_{1}(x)=x+1/2, \quad \mathcal{T}_{0}(x)=1.
\end{align*}
As mentioned earlier, in general, it is not easy to find the measure with respect to which $\mathcal{L}_n(x)$ and $\mathcal{T}_n(x)$ are orthogonal. However, it is shown in the following instance that for a particular choice of $\alpha_n$, this can be achieved. Assuming that the choice of $\rho_n$, $c_n$ and $d_n$ is same as given above, \eqref{condition 3} in \Cref{Theorem 2 with aplha_n condition} gives $\alpha_{n}=i/2$. For $\omega = \pm 1$, from \eqref{Q_n alpha beta} and \eqref{P_n to Q_n connection}, we get
\begin{align*}
	\mathcal{P}_n(x)= (-i)^n \sum_{k=0}^{n}(-n-1)_k \left(\dfrac{1-ix}{2}\right)^{k},
\end{align*}
using the power series expansion for the hypergeometric function, which is orthogonal with respect to the measure $w(x)=\dfrac{1}{\pi(1+x^2)}$. Therefore, a routine calculation implies that the self perturbed polynomials $\mathcal{L}_n(x)$ are given as
\begin{align*}
\mathcal{L}_n(x)=\mathcal{P}_n(x)-\alpha_n\mathcal{P}_{n-1}(x)=	\dfrac{(x+i)}{2i^{n+1}} \sum_{k=0}^{n-1}(-n)_k \left(\dfrac{1-ix}{2}\right)^{k},
\end{align*}
whose measure of orthogonality is $w'(x)=\dfrac{4}{\pi(1+x^2)^2}$.

%\section*{Declarations}
%{\noindent\bf Availability of data and material. }
%Data sharing not applicable to this article as no data sets were generated or analyzed during the current study.
%\\
%{\noindent\bf Competing interests.}
%The authors declare that they have no competing interests.
%\\
%{\noindent\bf Funding.} Not Applicable.
%\\
%{\noindent\bf Authors' contributions.} All authors contributed equally to this work.
%\\
{\noindent\bf Acknowledgements.}
The work of the second author is supported by the NBHM(DAE) Project No. NBHM/RP-1/2019.
%The authors wish to thank the anonymous referee for the careful reading of the manuscript and for constructive
%comments, observations, and suggestions that has helped in significantly improving the manuscript.
%%\subsection*{Code availability} Not Applicable.


\begin{thebibliography}{55}
	
%\bibitem{Alfaro Marcellan Pena Rezola JAMA 2003}
%M. Alfaro, F. Marcellán, A. Peña\ and\ M.L. Rezola, On linearly related orthogonal polynomials and their functionals, J. Math. Anal. Appl. {\bf 287} (2003), no.~1, 307--319.

\bibitem{Alfaro Marcellan Pena Rezola JCAM 2010}
M. Alfaro, F. Marcellán, A. Peña\ and\ M.L. Rezola, When do linear combinations of orthogonal polynomials yield new sequences of orthogonal polynomials?, J. Comput. Appl. Math. {\bf 233} (2010), no.~6, 1446--1452.

\bibitem{Andrews Askey Roy book 1999}
G. E. Andrews, R. Askey\ and\ R. Roy, {\it Special functions}, Encyclopedia of Mathematics and its Applications, 71, Cambridge University Press, Cambridge, 1999.

\bibitem{Barrios Ardila RACSAM 2020}
D. Barrios Rolan\'{\i}a\ and\ J. C. Garc\'{\i}a-Ardila, Geronimus transformations for sequences of $d$-orthogonal polynomials, Rev. R. Acad. Cienc. Exactas F\'{\i}s. Nat. Ser. A Mat. RACSAM {\bf 114} (2020), no.~1, Paper No. 26, 14 pp.

\bibitem{Batelo Bracciali Ranga JCAM 2005}
M. A. Batelo, C. F. Bracciali\ and\ A. Sri Ranga, On linear combinations of $L$-orthogonal polynomials associated with distributions belonging to symmetric classes, J. Comput. Appl. Math. {\bf 179} (2005), no.~1-2, 15--29.

\bibitem{Beardon Driver JAT 2005}
A. F. Beardon\ and\ K. A. Driver, The zeros of linear combinations of orthogonal polynomials, J. Approx. Theory {\bf 137} (2005), no.~2, 179--186.

\bibitem{Beckermann Derevyagin Zhedanov linear pencil 2010}	
B. Beckermann, M. Derevyagin\ and\ A. Zhedanov, The linear pencil approach to rational interpolation, J. Approx. Theory {\bf 162} (2010), no.~6, 1322--1346.

\bibitem{KKB Swami 2018}
K. K. Behera\ and\ A. Swaminathan, Biorthogonality and para-orthogonality of $R_I$ polynomials, Calcolo {\bf 55} (2018), no.~4, Paper No. 41, 22 pp.

\bibitem{KKB Swami PAMS 2019}
K. K. Behera\ and\ A. Swaminathan, Biorthogonal rational functions of $R_{II}$-type, Proc. Amer. Math. Soc. {\bf 147} (2019), no.~7, 3061--3073.

\bibitem{KKB LAA 2021}	
K. K. Behera, A generalized inverse eigenvalue problem and $m$-functions, Linear Algebra Appl. {\bf 622} (2021), 46--65.

\bibitem{Bracciali Pereira ranga 2020}
C. F. Bracciali, J. A. Pereira\ and\ A. S. Ranga, Quadrature rules from a $R_{II}$ type recurrence relation and associated quadrature rules on the unit circle, Numer. Algorithms {\bf 83} (2020), no.~3, 1029--1061.

\bibitem{Brezinski Driver Redivo-Zaglia ANM 2004}
C. Brezinski, K. A. Driver\ and\ M. Redivo-Zaglia, Quasi-orthogonality with applications to some families of classical orthogonal polynomials, Appl. Numer. Math. {\bf 48} (2004), no.~2, 157--168.

\bibitem{Bueno Marcell LAA 2004}
M. I. Bueno\ and\ F. Marcell\'{a}n, Darboux transformation and perturbation of linear functionals, Linear Algebra Appl. {\bf 384} (2004), 215--242.

\bibitem{Chen Srivastava Wang RACSAM 2021}
C.-P. Chen, H. M. Srivastava\ and\ Q. Wang, A method to construct continued-fraction approximations and its applications, Rev. R. Acad. Cienc. Exactas F\'{\i}s. Nat. Ser. A Mat. RACSAM {\bf 115} (2021), no.~3, Paper No. 97, 26 pp.

\bibitem{Chihara quasi PAMS 1957}
T. S. Chihara, On quasi-orthogonal polynomials, Proc. Amer. Math. Soc. {\bf 8} (1957), 765--767.

\bibitem{Chihara book 1978}
T. S. Chihara, {\it An Introduction to Orthogonal Polynomials}, Gordon and Breach Science Publishers, New York, 1978.

\bibitem{Derevyagin Derkach 2011}
M. Derevyagin\ and\ V. Derkach, Darboux transformations of Jacobi matrices and Pad\'{e} approximation, Linear Algebra Appl. {\bf 435} (2011), no.~12, 3056--3084.

\bibitem{Draux Integral transform 2016}
A. Draux, On quasi-orthogonal polynomials of order $r$, Integral Transforms Spec. Funct. {\bf 27} (2016), no.~9, 747--765.

\bibitem{Driver Jordaan Mbuyi ANM 2009}
K. Driver, K. Jordaan\ and\ N. Mbuyi, Interlacing of zeros of linear combinations of classical orthogonal polynomials from different sequences, Appl. Numer. Math. {\bf 59} (2009), no.~10, 2424--2429.

\bibitem{Shabrawy Shindy RACSAM 2020}
S. R. El-Shabrawy\ and\ A. M. Shindy, Spectra of the constant Jacobi matrices on Banach sequence spaces, Rev. R. Acad. Cienc. Exactas F\'{\i}s. Nat. Ser. A Mat. RACSAM {\bf 114} (2020), no.~4, Paper No. 182, 23 pp.

\bibitem{Fejer 1933}
L. Fej\'{e}r, Mechanische Quadraturen mit positiven Cotesschen Zahlen, Math. Z. {\bf 37} (1933), no.~1, 287--309.

\bibitem{Golub Van Loan 1996}
G. H. Golub\ and\ C. F. Van Loan, {\it Matrix computations}, third edition, Johns Hopkins Studies in the Mathematical Sciences, Johns Hopkins University Press, Baltimore, MD, 1996.

\bibitem{Hendriksen Njaastad Rocky J 1991}
E. Hendriksen\ and\ O. Nj\aa stad, Biorthogonal Laurent polynomials with biorthogonal derivatives, Rocky Mountain J. Math. {\bf 21} (1991), no.~1, 301--317.

%\bibitem{Esmail book}
%M. E. H. Ismail, {\it Classical and quantum orthogonal polynomials in one variable}, Encyclopedia of Mathematics and its Applications, 98, Cambridge University Press, Cambridge, 2005.

\bibitem{Esmail masson JAT 1995}
M. E. H. Ismail\ and\ D. R. Masson, Generalized orthogonality and continued fractions, J. Approx. Theory {\bf 83} (1995), no.~1, 1--40.

\bibitem{Esmail Ranga 2018}
M. E. H. Ismail\ and\ A. Sri Ranga, $R_{II}$ type recurrence, generalized eigenvalue problem and orthogonal polynomials on the unit circle, Linear Algebra Appl. {\bf 562} (2019), 63--90.

%\bibitem{Jesus Petronilho JAMA 2008}
%M. N. de Jesus\ and\ J. Petronilho, On linearly related sequences of derivatives of orthogonal polynomials, J. Math. Anal. Appl. {\bf 347} (2008), no.~2, 482--492.

\bibitem{Jordaan Tookos 2009}
K. Jordaan\ and\ F. To\'{o}kos, Interlacing theorems for the zeros of some orthogonal polynomials from different sequences, Appl. Numer. Math. {\bf 59} (2009), no.~8, 2015--2022.

\bibitem{Joulak ANM 2005}
H. Joulak, A contribution to quasi-orthogonal polynomials and associated polynomials, Appl. Numer. Math. {\bf 54} (2005), no.~1, 65--78.

\bibitem{Koekoek Lesky Swarttouw book 2010}
R. Koekoek, P. A. Lesky\ and\ R. F. Swarttouw, {\it Hypergeometric orthogonal polynomials and their $q$-analogues}, Springer Monographs in Mathematics, Springer-Verlag, Berlin, 2010.

\bibitem{Konhauser JAMA 1965}
J. D. E. Konhauser, Some properties of biothogonal polynomials, J. Math. Anal. Appl. {\bf 11} (1965), 242--260.

\bibitem{Kuijlaars McLaughlin 2005}
A. B. J. Kuijlaars\ and\ K. T.-R. McLaughlin, A Riemann-Hilbert problem for biorthogonal polynomials, J. Comput. Appl. Math. {\bf 178} (2005), no.~1-2, 313--320.

\bibitem{Kwon Lee_Paco Park 2001}
K. H. Kwon, D.W. Lee, F. Marcell\'{a}n\ and\ S.B. Park , On kernel polynomials and self-perturbation of orthogonal polynomials, Ann. Mat. Pura Appl. (4) {\bf 180} (2001), no.~2, 127--146.

\bibitem{Lubinsky Sidi JCAM 2022}
D. S. Lubinsky\ and\ A. Sidi, Some biorthogonal polynomials arising in numerical analysis and approximation theory, J. Comput. Appl. Math. {\bf 403} (2022), Paper No. 113842, 13 pp.

\bibitem{Marcellan Chaggara Ayadi 2021}
F. Marcell\'{a}n, H. Chaggara\ and\ N. Ayadi, 2-Orthogonal polynomials and Darboux transformations. Applications to the discrete Hahn-classical case, J. Difference Equ. Appl. {\bf 27} (2021), no.~3, 431--452.

\bibitem{Marcellan Peherstorfer Steinbauer 1996}
F. Marcell\'{a}n, F. Peherstorfer\ and\ R. Steinbauer, Orthogonality properties of linear combinations of orthogonal polynomials, Adv. Comput. Math. {\bf 5} (1996), no.~4, 281--295.

\bibitem{Marcellan Saib 2019}
F. Marcell\'{a}n\ and\ A. Saib, Linear combinations of $d$-orthogonal polynomials, Bull. Malays. Math. Sci. Soc. {\bf 42} (2019), no.~5, 2009--2038.

\bibitem{Finkelshtein Ribeiro Ranga Tyaglov PAMS 2019}
A. Mart\'{\i}nez-Finkelshtein, L.L. Silva Ribeiro, A. Sri Ranga\ and\ M. Tyaglov, Complementary Romanovski-Routh polynomials: from orthogonal polynomials on the unit circle to Coulomb wave functions, Proc. Amer. Math. Soc. {\bf 147} (2019), no.~6, 2625--2640.

\bibitem{Meurant 1992}
G. Meurant, A review on the inverse of symmetric tridiagonal and block tridiagonal matrices, SIAM J. Matrix Anal. Appl. {\bf 13} (1992), no.~3, 707--728.

\bibitem{Finkelshtein Ribeiro Ranga Tyaglov CRR 2020}
A. Mart\'{\i}nez-Finkelshtein, L.L. Silva Ribeiro, A. Sri Ranga\ and\ M. Tyaglov , Complementary Romanovski-Routh polynomials, orthogonal polynomials on the unit circle, and extended Coulomb wave functions, Results Math. {\bf 75} (2020), no.~1, Paper No. 42, 23 pp.

\bibitem{Raposo Weber Castillo Kirchbach 2007}
A.P. Raposo, H.J. Weber, D.E. Alvarez-Castillo\ and\ M. Kirchbach, Romanovski polynomials in selected physics problems, Cent. Eur. J. Phys. {\bf 5} (2007), 253-284.

%\bibitem{Petronilho JAMA 2006}
%J. Petronilho, On the linear functionals associated to linearly related sequences of orthogonal polynomials, J. Math. Anal. Appl. {\bf 315} (2006), no.~2, 379--393.

\bibitem{Riesz 1923}
M. Riesz, Sur le problème des moments, Troisième Note. Ark. Mat. Fys. 17, 1–52, (1923).

\bibitem{Shohat TAMS 1937}
J. Shohat, On mechanical quadratures, in particular, with positive coefficients, Trans. Amer. Math. Soc. {\bf 42} (1937), no.~3, 461--496.

\bibitem{swami vinay R2 2022}
V. Shukla\ and\ A. Swaminathan,  Chain sequences and Zeros of a perturbed $R_{II}$ type recurrence relation, arXiv: 2201.09409, 23 pages, 2022.

\bibitem{Silva Ranga LAA 2005}
A. P. da Silva\ and\ A. Sri Ranga, Polynomials generated by a three term recurrence relation: bounds for complex zeros, Linear Algebra Appl. {\bf 397} (2005), 299--324.

\bibitem{Srivastava biorthogonal Laguerre 1982}
H. M. Srivastava, Some biorthogonal polynomials suggested by the Laguerre polynomials, Pacific J. Math. {\bf 98} (1982), no.~1, 235--250.

\bibitem{Varma lmaz zarslan RACSAM 2019}
S. Varma, B. Y\i lmaz Ya\c{s}ar\ and\ M. A. \"{O}zarslan, Hahn-Appell polynomials and their $d$-orthogonality, Rev. R. Acad. Cienc. Exactas F\'{\i}s. Nat. Ser. A Mat. RACSAM {\bf 113} (2019), no.~3, 2127--2143.

\bibitem{Weber 2007}
H. J. Weber, Connections between Romanovski and other polynomials, Cent. Eur. J. Math. {\bf 5} (2007), no.~3, 581--595.

\bibitem{Zhedanov biorthogonal JAT 1999}
A. Zhedanov, Biorthogonal rational functions and the generalized eigenvalue problem, J. Approx. Theory {\bf 101} (1999), no.~2, 303--329.













\end{thebibliography}
\end{document}